\documentclass[11pt, oneside]{amsart}   
\usepackage[utf8]{inputenc}
\usepackage{amsmath}
\usepackage{amsfonts}
\usepackage{amssymb}
\usepackage{amsthm}
\usepackage{bbm} %For bold 1
\usepackage{graphicx}
\usepackage{tikz-cd}
\usepackage[margin=1in]{geometry}
\usepackage{outlines}
\usepackage[hidelinks,hypertexnames=false]{hyperref}			%Links for references, crossreferences
\hypersetup{
	colorlinks=false
}

\author{Lindsay Dever}
\address{Millersville University, Department of Mathematics, PO Box 1002,
	Millersville, PA 17551, USA}
\email{Lindsay.Dever@millersville.edu}
\address{Bryn Mawr College, Department of Mathematics, 101 N Merion Ave,
	Bryn Mawr, PA 19010, USA}

\title{Bias in the Distribution of Holonomy}

\newcommand\PSL{\mathrm{PSL}_2\mathbb{C}}

\newcommand\gtildehat{\hat{ \tilde g}_{y, \eta}}

\newcommand\PSU{\mathrm{PSU}}
\newcommand\GL{\mathrm{GL}}

\newcommand{\p}[2]{\frac{\partial #1}{\partial #2}}
\newcommand\vol{\mathop{\mathrm{vol}}}
\newcommand\Ei{\mathop{\mathrm{Ei}}}
\newcommand\hol{\mathop{\mathrm{hol}}}

\newcommand\sump{\sideset{}{^P}\sum}
\newcommand\sgn{\text{sgn}}

\newcommand\feven{f_{\text{even}}}
\newcommand\fodd{f_{\text{odd}}}

\newtheorem{theorem}{Theorem}[section]
\newtheorem{lemma}[theorem]{Lemma}
\newtheorem{prop}[theorem]{Proposition}
\newtheorem{cor}[theorem]{Corollary}

\theoremstyle{definition}
\newtheorem{remark}[theorem]{Remark}

\title[Bias in the distribution of holonomy]{Bias in the distribution of holonomy\\  on compact hyperbolic 3-manifolds}

\begin{document}

\begin{abstract} 
Ambient prime geodesic theorems provide an asymptotic count of closed geodesics by their length and holonomy and imply effective equidistribution of holonomy. We show that for a smoothed count of closed geodesics on compact hyperbolic 3-manifolds, there is a persistent bias in the secondary term which is controlled by the number of zero spectral parameters. In addition, we show that a normalized, smoothed bias count is distributed according to a probability distribution, which we explicate when all distinct, non-zero spectral parameters are linearly independent. Finally, we construct an example of dihedral forms which does not satisfy this linear independence  condition.
\end{abstract}


\maketitle

\section{Introduction and Background}

\subsection{Distribution of holonomy}

The Prime Geodesic Theorem for a hyperbolic 3-manifold $\Gamma \backslash \mathbb{H}^3$ states that the count $\pi_\Gamma(y)$ of primitive geodesics of length up to $y$ satisfies
\[\pi_\Gamma(y) = \text{Ei}_\Gamma(y) + \mathrm{O}_{\Gamma, \epsilon}(e^{(\frac53+\epsilon)y}),\]
where 
\begin{equation*}
\label{ei-varpi}
\Ei\nolimits_{\Gamma}(y)=\int_2^y d\varpi_{\Gamma}(t),\quad d\varpi_{\Gamma}(t)=\bigg(\frac{e^{2t}}t+\smash[t]{\sum_{j=1}^k}\frac{e^{(1+\nu_j)t}}t\bigg)\,d t
\end{equation*}
 and $\{0 <\nu_k \leq ... \leq \nu_1< 1\}$ are the eigenvalues $1 - \nu_j^2$ of the hyperbolic Laplacian \cite[Theorem 5.1]{Sarnak1983}.

The closed geodesics on a hyperbolic 3-manifold are parametrized both by their hyperbolic length and holonomy, which is the angle of rotation from parallel transport along the geodesic; see \eqref{T-def} for a full definition. We can then count geodesics by their full parameter space of both length and holonomy. Given an interval $I \subseteq \mathbb{R}/2 \pi \mathbb{Z}$, we define $\pi_\Gamma(y, I)$ as the count of geodesics with length up to $y$ and holonomy in $I$. In the broader context of locally symmetric spaces of negative curvature and finite volume, Sarnak and Wakayama showed that holonomy is equidistributed; in the context of hyperbolic 3-manifolds, the count of primitive geodesics satisfies
\[\pi_\Gamma(y,I) \sim_\Gamma \frac{|I|}{2\pi} \pi(y)\]
for any interval $I \subseteq \mathbb{R}/2 \pi \mathbb{Z}$  \cite[Theorem 1, Corollary 1]{SarWa}. See also \cite{MMO2014, MMO2014a} which improves the error term by a small, unspecified power savings in the more general setting of a geometrically finite, Zariski dense $\Gamma < \PSL$.
 
In the context of compact hyperbolic 3-manifolds, the above result was improved significantly in \cite[Theorem 6.5]{DM21}, which showed that
\[\pi_\Gamma(y,I) = \frac{|I|}{2\pi} \pi_\Gamma(y) + \mathrm{O}_\Gamma( e^{5y/3}).\]
This more refined count implies not only effective equidistribution, but also equidistribution in shrinking intervals of holonomy of size $\gg e^{-y(1/3-\delta)}$ for any $\delta>0$.

Each of these results uses a ``sharp'' count on both the length and holonomy. We can also consider a count of geodesics weighted by smooth functions $f: \mathbb{R}/2 \pi \mathbb{Z} \rightarrow \mathbb{R}$ on the holonomy and $g_{y, \eta}: \mathbb{R} \rightarrow \mathbb{R}$ on the length; this is  a natural step in the proof of sharp-cutoff results. To define  $g_{y, \eta}$, which approximates a sharp cutoff at $y$, we fix a smooth, even, non-negative function $\psi :\mathbb{R} \rightarrow \mathbb{R}$ which is compactly supported on $[-1,1]$ and has integral $1$,  and let  $g_{y, \eta} := \psi_\eta \star \chi_{[-y,y]}$, where $\psi_\eta(t) := \frac{1}{\eta}\psi\left(\frac{t}{\eta}\right)$ and $\chi_{[-y, y]}$ is the characteristic function on $[-y, y]$. For this smoothed count and $0<\eta\leq \eta_0$, we have the stronger asymptotic
\begin{equation}\label{smooth-APGT}
\begin{split}
\pi_\Gamma(f,g_{y, \eta}) & := \sump_{[\gamma]} f(\hol(\gamma)) g_{y, \eta}(\ell(\gamma))\\
&= \frac{1}{2\pi} \int_0^{2\pi} f(\theta)\, d\theta \cdot \int_2^{\infty} g_{y, \eta}(u) \, d \varpi_\Gamma(u) + \mathrm{O}_{\Gamma, \eta_0, f}\big(e^y \big),
\end{split}
\end{equation}
where the error term exhibits essentially square-root cancellation compared to the main term \cite[Proposition 6.1]{DM21}. 

A natural question following equidistribution of holonomy is whether there is bias in the secondary term. This can be captured by choosing a function $f$ on holonomy which averages to zero, which will cause the main term to disappear. For technical reasons, we also choose to weight the count by the length of the geodesic, defining
\begin{equation}\label{bias-count-def-intro}
L^P[f,g_{y, \eta}] := \sump_{[\gamma]} \ell(\gamma) f(\hol(\gamma)) g_{y, \eta}(\ell(\gamma))
\end{equation}
for an even, smooth function $f: \mathbb{R}/2 \pi \mathbb{Z}\rightarrow \mathbb{R}$ such that $\hat f(0) = 0$ and a smoothing parameter $\eta>0$; we fix $f$ and $\eta$ throughout this paper. We also normalize to $E[f,g_{y, \eta}] := e^{-y} L^P[f,g_{y, \eta}]$.

With this normalization, we can compute the average value of $E[f, g_{y, \eta}]$. The average depends on  multiplicities of the irreducible, unitary representations $\pi_{\nu,p}$ of $L^2(\Gamma\backslash G)$; see Section \ref{background} for details on these representations. In particular, the average value  depends on the multiplicities of zero spectral parameters.
\begin{theorem}\label{Eyaverage-intro} The average value of $E[f, g_{y, \eta}]$ as $y \rightarrow \infty$ is
\begin{equation}\label{b-f-eta-intro}
b_{f,\eta} = \left(\sum_{p \neq 0} m_\Gamma(\pi_{0, p}) \Re\big(\hat f(p)\big)  -  2 \Re\big(\hat f(1)\big)\right) 2 c_{0, \eta},
\end{equation}
where $m_{\Gamma}(\pi_{0,p})$ is the multiplicity of representations in $L^2(\Gamma \backslash G)$ with spectral parameter $0$ and non-zero integer parameter $p$ and $c_{0, \eta} = \int_{-\infty}^\infty \psi_\eta(t) e^t \, dt$ is a constant accounting for the choice of smooth cutoff function.
\end{theorem}
We refer to  $b_{f, \eta}$ as the bias constant. Note that the Schwartz decay of $\hat f$, together with the Weyl bound from Proposition \ref{spectral-bound}, guarantees that the above sum converges absolutely. If none of the spectral parameters are equal to zero, then $b_{f, \eta} = -4 \Re(\hat f(1) ) c_{0, \eta} = -\frac{2}{\pi} c_{0, \eta} \int_0^{2 \pi} f(\theta) \cos(\theta) \, d \theta$ gives a bias of holonomy towards the left side of the circle, i.e. a preference towards angles closer to $\pi$. For example, if $f(\theta) = \cos(\theta)$, then $b_{f, \eta} \approx -2$. However, the existence of zero spectral parameters could push the bias constant  to be zero, indicating no bias towards the left or right side of the circle, or positive, indicating a bias towards the right side of the circle i.e. angles closer to $0$.

In addition, we find that $E[f, g_{y, \eta}]$ is distributed according to a probability distribution as $y \rightarrow \infty$.

\begin{theorem}\label{Eydistribution-intro} For every $\eta>0$ and smooth $f: \mathbb{R}/2\pi \mathbb{Z} \rightarrow \mathbb{R}$, there exists a probability measure $\mu_{f,\eta}$ such that
	\[\lim_{Y \rightarrow \infty} \frac1Y \int_{\eta_0}^Y h(E[f,g_{y, \eta}]) \,dy = \int_{\mathbb{R}} h(x) \,d \mu_{f,\eta}(x) = \mu_{f,\eta}(h)\]
	for all Lipschitz functions $h: \mathbb{R} \rightarrow \mathbb{C}$.
\end{theorem}

Further, if non-zero spectral parameters are linearly independent, we can additionally prove that $\mu_{f, \eta}$ is distributed according to an explicit probability distribution function which depends on the spectral parameters of $L^2(\Gamma \backslash G)$.
\begin{theorem}\label{LI-Theorem-intro}  Suppose  all distinct non-zero spectral parameters of $L^2(\Gamma \backslash G)$  are linearly independent. The distribution $\mu_{f,\eta}$ from Theorem \ref{Eydistribution-intro} has $d \mu_{f,\eta}(x) = p_{f,\eta}(x) \, dx$ where the distribution function is
	\[ p_{f,\eta}(x) = \frac{1}{( 2\pi)^2} \int_{-\infty}^\infty \prod_{j=1}^\infty  J_0\left(\xi 2 m_{\Gamma}(\pi_{is_j,p_j})\hat f(-p_j)|c_{s_j, \eta}|\right) \cos(\xi (x- b_{f, \eta})) \, d \xi,\]
	where $J_0$ is a Bessel function of the first kind, $c_{s, \eta} = \int_{-\infty}^\infty  \psi_\eta(t) \frac{e^{t(1+is)}}{1+is} \, dt$, and $b_{f, \eta}$ is as in \eqref{b-f-eta-intro}. The distribution $\mu_{f,\eta}$ is symmetric about $b_{f,\eta}$ and a bias exists if and only if $b_{f, \eta} \neq 0$.
\end{theorem}
Thus, when all distinct non-zero spectral parameters are linearly independent,  we can fully understand the distribution of $E[f, g_{y, \eta}]$ as $y \rightarrow \infty$.
%Thus, in what might be considered the typical case when there are no spectral parameters equal to zero, a bias exists and is controlled by the integral $\int_0^{2\pi} f(\theta) \cos(\theta)\, d\theta$. This indicates a bias towards the left side of the circle, i.e. a preference towards angles closer to $\pi$.

%Add more exposition abou what these theorems mean and why they indicate a bias.

\subsection{Chebyshev's bias}

One of the most famous instances of bias is Chebyshev's bias in the prime number race. The question was posed when Chebyshev observed that there are typically more primes congruent to $3 \bmod 4$ than $1\bmod 4$; in other words, if we define
\[\pi(x;q,a) = \# \text{ of primes } p \leq x: p \equiv a \bmod q,\]
then $\pi(x;4,3)-\pi(x;4,1)$ is typically positive. We can also consider the set
\[P_{4;3,1} = \{x \geq 0: \pi(x;4,3)>\pi(x; 4,1)\}\]
of values where $3 \bmod 4$ is ``winning the primes race''. Dirichlet's theorem states that there are infinitely many primes in each congruence class modulo $4$, and Littlewood proved in 1914 that $P_{4;3,1}$ and $P_{4;1,3}$ are both of infinite size \cite{Dirichlet, Littlewood}. In 1962, Knapowski and Tur\'an showed that this is also true for many $P_{q; a_1, a_2}$ with $(a_i, q) = 1$ \cite{KnapTuran}. They conjectured that the density of $P_{4;3,1}\cap [0,x]$ converges to $100\%$ as $x \rightarrow \infty$, but the current numerical evidence for large primes does not support this; see, for example, \cite{GM06}.

Rubinstein and Sarnak revolutionized this question by looking at the log density of $P_{q;a_1,...,a_n}$. In addition, they considered the normalized bias count
\[E(y) = \frac{\log x}{\sqrt{x}}(\varphi(q) \pi(x,q,a) - \pi(x)).\]
They showed that $E(y)$ is distributed according to a probability distribution and  that there is always a bias towards non-residue classes $a \bmod q$. Assuming Generalized Riemann Hypothesis and Grand Simplicity Hypothesis, Rubinstein and Sarnak explicitly calculated the distribution of $E(y)$ and computed the log densities of $P_{q;N, R}$ to $4$ digits for $q = 3, 4, 5, 7, 11, 13$ \cite{RS}.

We prove similar results for the distribution of holonomy about closed geodesics in compact hyperbolic 3-manifolds. The normalized bias count $E(y)$ can be compared to $E[f, g_{y, \eta}]$, and a probability distribution can similarly be obtained. With the additional assumption of linear independence of distinct spectral parameters, analogous to the Grand Simplicity Hypothesis for zeroes of L-functions, we obtain a specific formula for the probability distribution function. For non-spherical Maass forms, the analogue of the Riemann hypothesis is true since every spectral parameter $\nu$ falls on the ``critical line'' where $\Re(\nu)=0$.  If sufficiently many spectral parameters for a particular manifold were computed with some precision, the densities when $E[f, g_{y, \eta}]$ is positive or negative could  be calculated.

The source of Chebyshev's bias is the contribution of powers of primes (specifically squares of primes), which leads to a bias towards quadratic non-residues. In the case of bias in holonomy, the non-primitive geodesics actually have no contribution to the bias; the bias instead stems entirely from appearance of the Weyl discriminant in the orbital integral and trivial representation terms in the trace formula and the spectral terms for representations of $L^2(\Gamma \backslash G)$ with spectral parameter equal to zero. The count of bias in holonomy is established using the trace formula, rather than Dirichlet's theorem for arithmetic progressions, and so the primes race is a source of inspiration rather than a direct analogue.

There are many other questions stemming from the question of Chebyshev's bias. Recent work has focused on finite fields \cite{Cha08}, elliptic curves over finite fields \cite{CFJ16}, and certain number fields \cite{FJ20}. For a complete history of Chebyshev's bias, see \cite{GM06}. For prime geodesics on a compact hyperbolic surface, Philips and Sarnak showed that geodesics are equidistributed among homology classes; in addition, Rubinstein and Sarnak comment that a careful analysis of the secondary term indicates a strong bias towards the zero homology class \cite{PS87,RS}. Thus after considering equidistribution of holonomy for prime geodesics, we turn our attention to bias in the distribution of holonomy.

\subsection{Overview}

In this section, we give an overview of the proof for the sake of the reader, omitting details.

The main results, Theorems \ref{Eydistribution-intro} and \ref{LI-Theorem-intro}, stem from approximations of the smoothed bias count \eqref{bias-count-def-intro}. The ambient prime geodesic theorems give a ``trivial'' bound of $\mathrm{O}_\Gamma(e^y)$; however, to show that a bias exists, we need the full power of the non-spherical trace formulas, which result from the spectral decomposition of $L^2(\Gamma \backslash G)$. Using a preliminary version of Selberg's trace formula explicated by Lin and Lipnowski \cite{LL21}, we prove the following separable trace formula in Proposition \ref{separable-trace} for any smooth $f: \mathbb{R}/2 \pi \mathbb{Z} \rightarrow \mathbb{C}$ and any smooth, compactly-supported $g: \mathbb{R} \rightarrow \mathbb{R}$ such that $g(u) f(\theta) = g(-u) f(-\theta)$:
\begin{multline}\label{separable-trace-overview}
\sum_{[\gamma]} \ell(\gamma_0) w(\gamma)  g(\ell(\gamma)) f(\hol(\gamma))= \sum_{\nu,p} m_\Gamma(\pi_{\nu,p}) \hat f(-p) \hat  g\left( \frac{i \nu}{2 \pi}\right) +\frac{1}{2} \hat f(0) \int_{\mathbb{R}} (e^u + e^{-u}) g(u) \, du\\
 -  \hat g(0) \left(\frac{\hat f(1) + \hat f(-1)}{2}\right)+  \frac{1}{2 \pi} \vol(\Gamma \backslash G)(g''(0) f(0) + g(0) f''(0)),
 \end{multline}
where $w(\gamma)= |1-e^{\mathbb{C}\ell(\gamma)}|^{-1} |1-e^{-\mathbb{C}\ell(\gamma)}|^{-1} = (e^{\ell(\gamma)}+ e^{-\ell(\gamma)}- 2 \cos(\hol(\gamma)))^{-1}$ is from the Weyl discriminant, $\gamma_0$ is the primitive geodesic that generates $\gamma$, and $m_{\Gamma}(\pi_{\nu,p})$ is the multiplicity of the representation $\pi_{\nu,p}$ in $L^{2}(\Gamma \backslash G)$. 

We choose a smooth function $f$ with the property that $\hat f(0) = 0$; in other words, $f$ averages to zero. The trace  formula \eqref{separable-trace-overview} allows us to sample the length and holonomy of geodesics separately; the multiplicities that appear are principal series representations with spectral parameter $\nu = is$, where $s$ is real. Using  the function $\tilde g_{y, \eta}(u) = (e^u + e^{-u}) g_{y, \eta}(u)$, where $g_{y, \eta}$ is defined as in \eqref{g-y-eta-def} (as well as odd functions $\tilde h_{y, \eta}$ and $h_{y, \eta}$) for $0< \eta \leq \eta_0<y$, in Proposition \ref{geom-trace-est} we obtain an estimate for the geometric sum
\begin{align*}
\tilde L[f, g_{y, \eta}] &: = \sum_{[\gamma]} \ell(\gamma_0) w(\gamma) (e^{\ell(\gamma)} + e^{-\ell(\gamma)})g_{y, \eta}(\ell(\gamma)) f(\hol(\gamma))\\
&= 2 e^y \sum_{s,p} m_\Gamma(\pi_{is,p}) \Re\left( \hat f(-p) e^{isy} c_{s, \eta}\right) - (\hat f(1) + \hat f(-1)) e^y c_{0, \eta}+ \mathrm{O}_{\Gamma, f, \eta_0}\left(\frac{1}{\eta^2} + 1\right),
\end{align*}
where $c_{s, \eta} = \int_{-\infty}^\infty  \psi_\eta(t) \frac{e^{t(1+is)}}{1+is} \, dt$. The spectral terms for $s \neq 0$ are oscillating and of size $e^y$ with rapid decay as $|s|, |p| \rightarrow \infty$. The following term contributes $\asymp e^y$ but is non-oscillating, and hence (along with any $s = 0$ terms) will contribute to a bias in holonomy.

We then must account for the change of weights and the removal of non-primitive geodesics to obtain an estimate of $L^P[f,g_{y, \eta}]$.
%\[L[f, g_{y, \eta}]: = \sum_{[\gamma]} \ell(\gamma_0) g_{y, \eta}(\ell(\gamma)) f(\hol(\gamma)),\]
%where the sum is over all geodesics (both primitive and imprimitive).
  In Lemma \ref{non-primitive-bound}, we show that the non-primitive geodesics  contribute $\ll_{\Gamma,f, \eta} e^{y/2}$. When changing weights in Lemma \ref{change-of-weight}, we obtain an additional contribution of bias from the trivial representation and $s = 0$ spectral terms, which is of the same size and direction as the bias terms above. Then we are able to obtain in Proposition \ref{bias-count} that when $\hat f(0) = 0$,
\begin{equation*}
L^P[f, g_{y, \eta}]  = 2 e^y \sum_{\substack{s\neq 0, p}} m_\Gamma(\pi_{is,p}) \Re\left( \hat f(-p) e^{isy}c_{s, \eta}\right)  + e^y b_{f,\eta}  + \mathrm{O}_{\Gamma, f, \eta_0}\left(e^{y/2}\left(1 + \frac{1}{\eta^2}\right)\right),
\end{equation*}
where 
\begin{equation}\label{b-f-eta-background}
b_{f,\eta} = \left(\sum_{p \neq 0} m_\Gamma(\pi_{0, p}) \Re(\hat f(p))  -  2 \Re(\hat f(1))\right) 2 c_{0, \eta}.
\end{equation}
 
The bias in holonomy is essentially coming from the secondary term in the Weyl discriminant (see Theorem \ref{LL}), which appears in the trace formula \eqref{separable-trace-overview} as the contribution of the trivial representation and the weight $w(\gamma)$ in the geometric terms. Essentially the main term $e^{\ell(\gamma)}$ in the Weyl discriminant contributes to the main term in the Ambient Prime Geodesic Theorems. Thus when this term is eliminated by choosing a function $f$ on holonomy with $\int_0^{2\pi} f(\theta) \, d\theta = 0$, the secondary  term results from the secondary term in the Weyl discriminant. The spectral terms are of competing size; however, these are oscillatory and centered at zero, so on average the bias persists. 

To fully understand the distribution of bias, we normalize to $E[f, g_{y, \eta}] = e^{-y} L^P[f, g_{y, \eta}]$. First, we estimate $E[f, g_{y, \eta}]$ by
\begin{equation*}
E^{(T)}[f, g_{y, \eta}] := 2 \sum_{\substack{|s|, |p|<T\\ s \neq 0}} m_\Gamma(\pi_{is,p}) \Re\left(\hat f(-p) e^{isy}c_{s, \eta}\right) +b_{f, \eta},
\end{equation*}
which has finitely many spectral terms. This is a function of pure exponentials, so in Lemma \ref{ETLemma} we  use the Kronecker-Weyl Theorem   to conclude that 
\[\lim_{Y \rightarrow \infty} \frac1Y \int_{\eta_0}^Y h\left(E^{(T)}[f, g_{y, \eta}]\right) \,dy = \int_A h_{T,f,\eta}(a) \,da = \int_{\mathbb{R}} h(x) \, d\mu_{T, f, \eta}(x)\]
where 
\[h_{T,f,\eta}(x_1,..,x_n) = h\left(2 \sum_{j =1}^n m_\Gamma(\pi_{is_j,p_j})  \Re\left(\hat f(-p_j) e^{2 \pi i x_j}c_{s_j, \eta}\right) +b_{f, \eta}\right),\]
 $A$ is the closure of the image of $\{\left(\frac{s_1 y}{2\pi},...,\frac{s_n y}{2\pi}\right), y \in \mathbb{R}\}$ in $\mathbb{T}^n = \mathbb{R}^n/\mathbb{Z}^n$,  $da$ is the Haar measure on $A$, and $d\mu_{T, f, \eta}$ is a probability measure. In the case when $s_1,..., s_n$ are linearly independent, then $A = \mathbb{T}^n$. In Theorem \ref{Eydistribution}, we show that the normalized bias count $E[f, g_{y, \eta}]$ is actually distributed according to a probability distribution, $\mu_{f, \eta}$, defined as the limit of $\mu_{T, f,\eta}$ as $T \rightarrow \infty$.

%In the case when all non-zero spectral parameters are linearly independent, the measures $\mu_{T, f, \eta}$ can be computed by an integral over the torus. In this case, $\mu_{f, \eta}$ can be computed by calculating the Fourier transform of $\mu_{T, f, \eta}$ using Lemma \ref{ETLemma} and taking the limit as $T \rightarrow \infty$. In Theorem  \ref{Eydistribution}, we then obtain the probability distribution function $p_{f, \eta}$ defining the probability measure $\mu_{f, \eta}$.

In Section \ref{LI-section}, we assume that  all distinct non-zero spectral parameters are linearly independent. In this case, the measures $\mu_{T, f, \eta}$ can be computed by an integral over the torus $\mathbb{T}^n$. In Proposition \ref{mu-hat-formula}, we use the exponential function $h_\xi(x) = e^{-i\xi x}$ in Lemma \ref{ETLemma} to obtain the Fourier transform $\hat \mu_{T, f, \eta}(\xi)$, and we compute that
\[\hat \mu_{f,\eta}(\xi) =  e^{-i \xi b_{f, \eta}} \prod_{j =1}^\infty J_0\left(\xi 2 m_{\Gamma}(\pi_{i {s_j}, p_j}) | \hat f(-p_j) c_{s_j,\eta}|\right)\]
 by taking the limit as $T \rightarrow \infty$.  Then in Theorem \ref{LI-Theorem}, we use Fourier inversion to obtain the probability distribution function $p_{f, \eta}$ defining the probability measure $\mu_{f, \eta}$. Finally, in Section \ref{LD-section}, we construct an example of dihedral forms where there is an arithmetic progression of spectral parameters; thus, the linear independence hypothesis is not met. It is an interesting question whether spectral parameters are linearly independent when there are no dihedral forms present.

\subsection{Acknowledgments} Thank you to Gergely Harcos and Kimball Martin for their generous and helpful insights in the construction of co-compact dihedral forms, and to my advisor, Djordje Mili\'cevi\'c. 

\subsection{Notation}

We write $f = \mathrm{O}(g)$ or $f\ll g$ to indicate that $|f| \leq C g$ for some constant $C>0$, which may change from line to line and is absolute unless indicated with a subscript. The discrete subgroup $\Gamma$ is fixed and various sums (including but not limited to $L^P[f, g_{y, \eta}]$, $E[f, g_{y, \eta}]$, and $E^{T}[f, g_{y, \eta}]$) and the measures $\mu_{f, \eta}$, $\mu_{T, f, \eta}$ have dependence on $\Gamma$ which is clear from the definitions. This dependence will be indicated in upper bounds but not in the notation of the named functions and measures.

\sloppy For a periodic function $f: \mathbb{R}/2 \pi \mathbb{Z} \rightarrow \mathbb{C}$, we denote the Fourier coefficient as $\hat f(p) = \frac{1}{2 \pi} \int_{0}^{2 \pi} f(\theta) e^{-i p \theta} \, d\theta$, and for a Schwartz function $f: \mathbb{R} \rightarrow \mathbb{C}$, we choose the Fourier transform $\hat f(\xi) = \int_{\mathbb{R}} f(x) e^{-2 \pi i \xi x} \, dx$.

\section{Background on Trace Formulas, Equidistribution, and Measure}

\subsection{Background on groups, geometry, and representations}\label{background} Let $\Gamma$ be a discrete, co-compact, torsion-free subgroup of $G$; then $\Gamma \backslash \mathbb{H}^3$ is a compact hyperbolic 3-manifold. In this section, we will introduce the group $\PSL$, connections to the geometry of hyperbolic 3-manifolds, and the representation space $L^2(\Gamma \backslash G)$.

There is a one-to-one correspondence between elements of the group $G = \PSL$ and isometries of the hyperbolic space $\mathbb{H}^3$; for a description of the isometries of $\mathbb{H}^3$, see \cite[\S1.1]{EGM}. The maximal compact subgroup of $G$ is 
\[K = 
\PSU_2=\left\{\begin{pmatrix}
\alpha & \beta\\
- \bar \beta & \bar \alpha
\end{pmatrix} : \alpha, \beta \in \mathbb{C},\, |\alpha|^2 + |\beta|^2 = 1\right\}/\{\pm 1\}\]
which acts on $\mathbb{H}^3$ by rotation about a particular geodesic. It turns out that there is a natural isomorphism $\mathbb{H}^3 \cong G/K$, and  the action of $G$ corresponds to left-multiplication of $G/K$ by $G$.

When $\Gamma$ is co-compact and torsion-free, every element $\gamma \in \Gamma$ is hyperbolic or loxodromic and is conjugate to a diagonal element
\begin{equation}\label{T-def}
t_\gamma\in T:=\left\{\begin{pmatrix}
e^{(u + i \theta)/2} & 0 \\
0 & e^{-(u + i \theta)/2}
\end{pmatrix}:u\in\mathbb{R},\,\theta\in\mathbb{R}/2\pi\mathbb{Z}\right\},
\end{equation}
which we may also refer to as $t_{u, \theta}$. The group $T$ of diagonal elements has Haar measure $du \,d\theta/2\pi$. The action of $\gamma$ leaves a single geodesic of $\mathbb{H}^3$ invariant; in the manifold $\Gamma \backslash \mathbb{H}^3$, this corresponds to a single closed geodesic. The parameters $u$ and $\theta$ are the length and holonomy, respectively, of the closed geodesic, which we will refer to as $\ell(\gamma)$ and $\hol(\gamma)$. The action of $\gamma$ is a shift along the closed geodesic by the length $\ell(\gamma)$, followed by a rotation around the geodesic by the holonomy $\hol(\gamma)$.

Using the Iwasawa decomposition of $G$, we can define a Haar measure on $G$; see \cite[\S 2.1]{DM21}, for example,  for an explicit description of this measure. Since $\Gamma$ is discrete, the Haar measure on $G$ also induces a Haar measure on the quotient space $\Gamma \backslash G$; we can then consider the space $L^2(\Gamma \backslash G)$ of square-integrable functions on $\Gamma \backslash G$. When $\Gamma$ is co-compact, there is a spectral decomposition into discrete representations of $G$,
\begin{equation}\label{L2-decomp}
L^2(\Gamma \backslash G) = \bigoplus_{\pi \in \hat G}  m_\Gamma(\pi) \pi,
\end{equation}
where $\hat G$ is the set of irreducible, unitary representations of $G$ and $m_\Gamma(\pi)$ is the multiplicity of $\pi$ in $L^2(\Gamma \backslash G)$. Since $\Gamma$ is co-compact, there are countably many representations with non-zero multiplicity. Each representation $\pi$, under the right $K$-action, further decomposes into  automorphic forms, each with a particular $K$-type; for a detailed description, see \cite[Chapter 8]{Lok}.

The representations of $G = \PSL$ are completely classified and are induced by the character
\begin{equation}
\label{characters-T}
\chi_{\nu,p}\Bigg(\begin{pmatrix}
e^{(u + i \theta)/2} & * \\
0 & e^{-(u + i \theta)/2}
\end{pmatrix}\Bigg) = e^{u \nu + i p \theta}
\end{equation}
on the Borel subgroup of $G$; this character can also be restricted to the group $T$ of diagonal elements. The irreducible, unitary representations are:
\begin{itemize}
\item the trivial representation,
\item unitary principal series representations $\pi_{\nu, p}$ with $\nu \in i \mathbb{R}$ and $p \in \mathbb{Z}$,
\item complementary series representations $\pi_{\nu, 0}$ for $0<\nu<1$.
\end{itemize}
and the only equivalences are $\pi_{\nu,p} \cong \pi_{-\nu,-p}$ \cite[Theorem 16.2]{Knapp}. For a complete description of the representations of $\PSL$, see \cite[\S II.4]{Knapp}, \cite[\S 2.3]{LL21}.

\subsection{Non-spherical trace formulas} Selberg's trace formula relates spectral to geometric information on the manifold $\Gamma \backslash \mathbb{H}^3$. In the classical case, eigenvalues of the Laplacian operator are related to the lengths of closed geodesics as well as the volume. However, to obtain information about the holonomy of closed geodesics, we need a non-spherical trace formula that involves multiplicities from the decomposition \eqref{L2-decomp}. The preliminary version of an explicated, non-spherical trace formula appears in \cite[Corollary B.7]{LL21}. 

\begin{theorem}[Lin--Lipnowski]\label{LL}
	Let $\Gamma<\PSL$ be a discrete, co-compact, torsion-free subgroup. Then, for every smooth, compactly-supported function $F:T\to\mathbb{C}$ such that $F(t) = F(t^{-1})$,	
	\begin{align*}
	&\sum\limits_{\nu,p} m_\Gamma(\pi_{\nu,p}) \hat{F}(\chi_{\nu,p}^{-1}) + \frac{1}{2} \int_T |D(t^{-1})|^{1/2} F(t) \, d t\\
	&\qquad= - \frac{1}{2\pi} \vol(\Gamma \backslash G) \Big(\p{^2}{u^2} + \p{^2}{\theta^2}\Big)F\Big|_{t = 1} + \sum\limits_{[\gamma]} \ell(\gamma_0) |D(t_\gamma^{-1})|^{-1/2} F(t_\gamma),
	\end{align*}
where $\hat F(\chi) = \int_T F(t) \chi^{-1}(t) \,d t$ is the Abel transform and $D(t_\gamma) = (1- e^{\mathbb{C} \ell(\gamma)})^2(1 - e^{- \mathbb{C} \ell(\gamma)})^2$ is the Weyl discriminant. The first sum is over the principal and complementary series representations appearing in $L^2(\Gamma \backslash G)$, and the last sum is over the non-trivial conjugacy classes of $\Gamma$. The length $\ell(\gamma_0)$ refers to the length of the primitive geodesic $\gamma_0$ which generates $\gamma$.
\end{theorem}

Choosing a separable function $F(t_{u,\theta}) = g(u) f(\theta)$ in the above equation allows us to futher explicate this formula and obtain information about the length and holonomy of closed geodesics.

\begin{prop}\label{separable-trace} For any discrete, co-compact, torsion-free subgroup $\Gamma < \PSL$, smooth, compactly-supported function $g:\mathbb{R} \rightarrow \mathbb{C}$, and smooth function $f: \mathbb{R}/2 \pi \mathbb{Z} \rightarrow \mathbb{C}$ such that $g(u) f(\theta) = g(-u) f(-\theta)$, the following trace formula holds:
\begin{align*}
\sum_{[\gamma]} \ell(\gamma_0) w(\gamma)  g(\ell(\gamma)) f(\hol(\gamma))&= \sum_{\nu,p} m_\Gamma(\pi_{\nu,p}) \hat f(-p) \hat  g( \frac{i \nu}{2 \pi})+\frac{1}{2} \hat f(0) \int_{-\infty}^{\infty} (e^u + e^{-u}) g(u) \, du\\
&-  \hat g(0) \left(\frac{\hat f(1) + \hat f(-1)}{2}\right)+  \frac{1}{2 \pi} \vol(\Gamma \backslash G)(g''(0) f(0) + g(0) f''(0)),
 \end{align*}
where $w(\gamma ) =  |1-e^{\mathbb{C}\ell(\gamma)}|^{-1} |1-e^{-\mathbb{C}\ell(\gamma)}|^{-1}  = e^{\ell(\gamma)} + e^{ - \ell(\gamma)} - 2 \cos(\hol(\gamma))$.

 If $f$ and $g$ are even, then this simplifies to 
\begin{multline*}
\sum_{[\gamma]} \ell(\gamma_0) w(\gamma)  g(\ell(\gamma)) f(\hol(\gamma)) = \frac12 \sum_{\nu,p} (m_\Gamma(\pi_{\nu,p}) + m_\Gamma(\pi_{\nu, -p})) \hat f(p) \hat  g( \frac{i \nu}{2 \pi})\\
+\frac{1}{2} \hat f(0) \int_{-\infty}^\infty (e^u + e^{-u}) g(u) \, du -  \hat g(0) \hat f(1)+  \frac{1}{2 \pi} \vol(\Gamma \backslash G)(g''(0) f(0) + g(0) f''(0)).
 \end{multline*}
 
 If $f$ and $g$ are odd, then this simplifies to
 \begin{equation*}
 \sum_{[\gamma]} \ell(\gamma_0) w(\gamma)  g(\ell(\gamma)) f(\hol(\gamma))=\frac12 \sum_{\nu,p} (m_\Gamma(\pi_{\nu,p})- m_\Gamma (\pi_{\nu, -p})) \hat f(-p) \hat  g( \frac{i \nu}{2 \pi}).
 \end{equation*}
\end{prop}

\begin{proof} For a separable smooth function $F(t_{u, \theta}) = g(u) f(\theta)$, where $f$ and $g$ are smooth and $g$ is compactly supported, we  explicate each term in the trace formula from Theorem \ref{LL}. For the spectral terms, we have:

\begin{align*}
\hat F(\chi_{\nu,p}^{-1}) &=  \int_T F(t_{u, \theta}) \chi_{\nu,p}^{-1}(t_{u, \theta}) dt = \frac{1}{2 \pi} \int_0^{2\pi} \int_{-\infty}^\infty g(u) f(\theta) e^{u \nu +ip\theta} \, du\, d\theta\\
& = \frac{1}{2 \pi} \int_{0}^{2 \pi} f(\theta) e^{i p \theta} \, d\theta  \int_{-\infty}^\infty g(u) e^{u \nu} \, du\\
& = \hat f(-p) \hat g( \frac{i \nu}{2 \pi}).
\end{align*}

Note that the modulus of the Weyl discriminant is
\begin{align*}
|D(t_{\gamma}^{-1})|^{1/2} &= |(1- e^{u+i\theta})(1- e^{-(u+i\theta)})|=e^{u} |1- e^{-(u+i \theta)}|^2\\
&= e^u (1-e^{-(u + i \theta)}) (1 - e^{ -u + i \theta})= e^u + e^{-u} - 2 \cos(\theta).
\end{align*}

From the trivial representation, we have the contribution:

\begin{align*}
\frac{1}{2} \int_T &|D(t^{-1})|^{1/2} F(t) \, d t\\
&=\frac{1}{4 \pi} \int_0^{2 \pi} \int_{-\infty}^\infty |(1- e^{u + i \theta})(1 - e^{ -(u + i \theta)})| F(t_{u, \theta}) \,du \, d\theta\\
& =\frac{1}{4 \pi} \int_0^{2\pi} \int_{-\infty}^{\infty} (e^u + e^{-u} - 2 \cos(\theta)) g(u) f(\theta) \, du \, d\theta\\
& = \frac{1}{2} \int_{-\infty}^{\infty} (e^u + e^{-u}) g(u) \, du \frac{1}{2 \pi} \int_0^{2\pi} f(\theta) \, d\theta - \frac{1}{2 \pi} \int_{-\infty}^\infty g(u) \, du \int_0^{2\pi} f(\theta) \cos(\theta) \, d \theta\\
& = \frac{1}{2} \hat f(0) \int_{-\infty}^{\infty} (e^u + e^{-u}) g(u) \, du - \hat g(0) \frac{\hat f(1) + \hat f(-1)}{2}.
\end{align*}

Note that if $f$ is even, then $\frac{\hat f(1) + \hat f(-1)}{2} = \hat f(1)$, and if $f$ is odd, then $\frac{\hat f(1) + \hat f(-1)}{2} = 0$.

For the non-trivial geometric terms, we have that 
\[\ell(\gamma_0)|D(t_\gamma^{-1})|^{-1/2} F_n(t_\gamma) = \ell(\gamma_0) (e^{\ell(\gamma)} + e^{-\ell(\gamma)}- 2 \cos(\hol(\gamma)))^{-1} g(\ell(\gamma)) f(\hol(\gamma)).\]

The identity element contributes
\[- \frac{1}{2 \pi}\vol(\Gamma \backslash G) \Big(\p{^2}{u^2} + \p{^2}{\theta^2}\Big)g(u) f(\theta)\Big|_{u = 0, \theta = 0} = -\frac{1}{2 \pi}\vol(\Gamma \backslash G)(g''(0) f(0) + g(0) f''(0)). \qedhere\]

\end{proof}

A consequence of the non-spherical trace formula is a bound on the spectral parameters within an interval of fixed length, which is a familiar step in the derivation of Weyl's law. We restrict to representations $\pi_{\nu,n}$ for a fixed integer $n$ and obtain a uniform bound; a proof of this proposition appears in \cite[Prop. 2.5]{DM21}.

\begin{prop}[Dever-Mili\'cevi\'c] \label{spectral-bound}
Let $\Gamma < \PSL$ be a discrete, co-compact, torsion-free subgroup. Then, for every $n \in \mathbb{Z}$ and $R \in \mathbb{R}$, the multiplicities $m_\Gamma(\pi_{\nu,n})$ of representations $\pi_{\nu, n}$ in $L^2(\Gamma \backslash G)$ are bounded by
\[\sum_{R-1\leq \nu \leq R+1} m_\Gamma(\pi_{\nu,n})\ll \vol(\Gamma \backslash G) \cdot (R^2 + n^2) + \mathrm{O}_\Gamma(1).\]
\end{prop}

This upper bound on the spectral parameters in fixed-length intervals allows us to use the non-spherical trace formula to obtain a bias count on the geometric side, while bounding the contribution from the spectral side.

\subsection{Kronecker-Weyl theorem} A critical observation in the proof of the distribution of 
\begin{equation*}
E^{(T)}[f, g_{y, \eta}] := 2 \sum_{\substack{|s|, |p|<T\\ s \neq 0}} m_\Gamma(\pi_{is,p}) \Re\left(\hat f(p) e^{isy}c_{s, \eta}\right) +b_{f, \eta}
\end{equation*}
is that this is a continuous function of a finite number of exponential functions $e^{isy}$, where $s$ is a non-zero spectral parameter of $L^2(\Gamma \backslash G)$. We can then apply the Kronecker-Weyl theorem to determine the distribution of this sum as $y \rightarrow \infty$.

We state the Kronecker-Weyl theorem in its full generality, which was written up formally by \cite{Devin} and based on \cite{Humphries}.

\begin{theorem}[Kronecker-Weyl]\label{Kronecker-Weyl}  Let $t_1,..., t_n$ be real numbers, and let $H$ be the closure of $\allowbreak \{ (e^{2 \pi i t_1 y}, ..., e^{2 \pi i t_n y}):  y \in \mathbb{R}\}$ in $\mathbb{T}^n$.
	
	1) $H$ is a $r$-dimensional subtorus of $\mathbb{T}^n$, where $r$ is dimension of the span of $t_1, ..., t_n$ over $\mathbb{Q}$.
	
	2) For any continuous function $h: \mathbb{T}^n \rightarrow \mathbb{C}$, we have that 
	\[\lim_{Y \rightarrow \infty} \frac{1}{Y} \int_{0}^Y h(e^{2 \pi i t_1 y},..., e^{2 \pi i t_n y}) \, dy = \int_H h(z) \, d \mu_H(z).\]
	where $\mu_H$ is the Haar measure on $H$. 
\end{theorem}
Often, the Kronecker-Weyl theorem is stated under the assumption that $t_1,..., t_n$ are linearly independent over $\mathbb{Q}$. In this case, $H = \mathbb{T}^n$ and $d\mu_H$ is the Haar measure $dx_1 \cdots \, dx_n$ on the torus $[0,1]^n$. We can think of $(e^{2 \pi i t_1 y}, ..., e^{2 \pi i t_n y})$ as being ``randomly'' distributed throughout $\mathbb{T}^n$. Note that the integral on the right hand side has no dependence on $t_1,..., t_n$; sampling $(e^{2 \pi i t_1 y}, ... ,e^{2 \pi i t_n y})$ with a continuous function $h$ is (in the limit) equivalent to sampling $h$ on the torus. 

In the general statement, $H$ is a sub-torus of $\mathbb{T}^n$, and we can think of $(e^{2 \pi i t_1 y}, ... , e^{2 \pi i t_n y})$ as being ``randomly'' distributed throughout its closure $H$. Then rather than integrating along this dense, one-dimensional subset, we can instead integrate over the sub-torus $H$ with the Haar measure $\mu_H$, significantly simplifying the integral.

\subsection{L\'evy's continuity theorem} We will also use a stronger version of L\'evy's continuity theorem, proven, for example, in \cite[Section 26, Corollary 1]{Bill}, which allows us to conclude, under certain conditions, that the limit of a sequence of measures is also a measure. In particular, we require that the Fourier transforms of the measure (referred to as characteristic functions in probability theory) converge to a function which is continuous at $0$.

\begin{theorem}[L\'evy's continuity theorem]\label{Levy} Suppose that for measures $\mu_n$, $\lim_{n \rightarrow \infty} \hat \mu_n (\xi) = \phi(\xi)$ for each $\xi$ and the limit function $\phi$ is continuous at $0$. Then there exists a measure $\mu$ such that $\mu_n$ converges in distribution to $\mu$, and $\hat \mu(\xi) = \phi(\xi)$.
	
\end{theorem} 

%Check that "characteristic function" matches Fourier transform. Yes - $\phi(t) := \int e^{itx} \mu(dx)$

We will use this theorem to conclude that the limit of a particular sequence of measures is in fact a measure. We also obtain the additional information that $\mu_n$ converges in law (or, equivalently, in distribution) to $\mu$. We provide  clarification on the meaning of this definition with the following theorem from  \cite[Section 25, Theorem 25.8]{Bill}.

\begin{theorem}The following conditions are equivalent:
\begin{enumerate}
	\item $\mu_n$ converges in law to $\mu$,
	\item $\lim_{n \rightarrow \infty} \int f\, d \mu_n  = \int f d \,\mu$ for every bounded, continuous, real function $f$,
	\item $\lim_{n \rightarrow \infty} \mu_n(A) \rightarrow \mu(A)$ for every Borel set $A$ with $\mu(\partial A) = 0$.
\end{enumerate}
	
\end{theorem}

\section{Count of Bias in Holonomy}

Let $\psi: \mathbb{R} \rightarrow \mathbb{R}$ be a fixed, smooth, even, non-negative function compactly supported on $[-1,1]$ such that the integral $\|\psi\|_1 =1$. For $\eta>0$, we define
\begin{equation}\label{psi-eta-def}
\psi_\eta(t) := \frac{1}{\eta}\psi\left(\frac{t}{\eta}\right),
\end{equation}
which is compactly supported on $[-\eta,\eta]$ with integral $1$.

We define the smooth cutoff functions as the convolutions
\begin{equation}\label{g-y-eta-def}
g_{y, \eta} = \psi_\eta \star \chi_{[-y,y]}, \quad h_{y, \eta} = \psi_\eta \star \left( \chi_{[-y,y]} \cdot \sgn\right)
\end{equation}
to approximate a length cutoff at $y>0$; note that these functions are equal on $[\eta, \infty)$.

For the holonomy parameter, we use a smooth function $f: \mathbb{R}/2 \pi \mathbb{Z} \rightarrow \mathbb{C}$ with the property $\hat f(0) = 0$. In this case, $f$ averages to $0$, and the main term in \eqref{smooth-APGT} disappears. Alternately, we could think of subtracting the average value of a function on holonomy so as to sample the secondary term. Thus we define the sum
\begin{equation}\label{bias-count-def}
L^P[f, g_{y, \eta}] := \sump_{[\gamma]} \ell(\gamma) g_{y, \eta}(\ell(\gamma) )f\left(\hol(\gamma)\right)
\end{equation}
which captures the bias in holonomy when sampled with the smooth function $f$. As an example, consider $f(\theta) = \cos(\theta)$; in this case, $L^P[f, g_{y, \eta}]$ measures the preference of holonomy towards  the left side of the circle (in a smoothed sense).

To estimate $L^P[f, g_{y, \eta}]$, we start with a similar function 
\begin{equation}\label{L-tilde-def}
\tilde L[f, g_{y, \eta}]:= \sum_{[\gamma]} \ell(\gamma_0) w(\gamma)\left(e^{\ell(\gamma)} + e^{-\ell(\gamma)}\right) g_{y, \eta}(\gamma),
\end{equation}
where the sum is over both primitive and imprimitive geodesics,  $\gamma_0$ is the primitive geodesic generating $\gamma$, and $w(\gamma)$ is defined as in Proposition \ref{separable-trace}. This can be estimated by the trace formula using the sampling functions $\tilde g_{y, \eta}(x) = (e^x + e^{-x}) g_{y, \eta}(x)$ and $\tilde h_{y, \eta}(x) = (e^x + e^{-x}) h_{y, \eta}(x)$. The exponential growth in these functions will account for  the exponentially decaying weight $w(\gamma)$ in the trace formula. The estimate will contain constants 
\begin{equation}\label{c-s-eta-def}
c_{s, \eta} = \int_{-\infty}^\infty \psi_\eta(t) \frac{e^{t(1+is)}}{1+is} \, dt
\end{equation}
which depend on the spectral parameter $s$ and the smoothing parameter $\eta$. We now remark on the Schwartz decay of $c_{s, \eta}$.

\begin{remark}\label{cs-decay} Fixing $\eta$, we have that $c_{s, \eta} = \frac{\hat j_\eta(-s/2 \pi)}{1 + is}$ where $j_\eta(t)=\psi_\eta(t)e^t$. This results in the Schwartz decay of $c_{s, \eta}$ (for a fixed $\eta>0$).  If we make the $\eta$ dependence explicit, then we have that
	\[c_{s, \eta} = \frac{1}{\eta(1 + is)} \int_{-\infty}^\infty  \psi(t/\eta)e^t e^{ist} \, dt = \frac{1}{1+is} \int_{-\infty}^\infty  \psi(t) e^{\eta t} e^{i\eta s t} \, dt= \frac{\hat \psi(\frac{\eta(i-s)}{2 \pi})}{1+is}.\]
	Since $\psi$ is smooth and compactly supported, we can then use the relationship $\hat \psi(\xi)  = \frac{\widehat{\psi^{(k)}}(\xi)}{(2 \pi i \xi)^k}$ and the Schwartz decay of $\widehat{\psi^{(k)}}$ to calculate that for $0<\eta<\eta_0$,
	\[c_{s, \eta} = \frac{\hat \psi(\frac{\eta(i-s)}{2 \pi})}{1+is} \ll_{k} \widehat{\psi^{(k)}}(\frac{\eta(i-s)}{2\pi})\frac{1}{(|s|+1)(|\eta s|^k+1)} \ll_{k, \eta_0} \frac{1}{(|s|+1)(|\eta s|^k+1)} \]
	since $ \widehat{\psi^{(k)}}$ is bounded in a vertical strip.
\end{remark}

\begin{prop}\label{geom-trace-est} For any discrete, co-compact, torsion-free subgroup $\Gamma < \PSL$, $y> \eta_0 \geq \eta>0$, and any smooth function $f : \mathbb{R}/2 \pi \mathbb{Z} \rightarrow \mathbb{R}$ such that $\hat f(0) = 0$,
\begin{equation*}
\tilde L[f, g_{y, \eta}] = 2 e^y \sum_{s,p} m_\Gamma(\pi_{is,p}) \Re\big( \hat f(-p) e^{isy} c_{s, \eta}\big) - (\hat f(1) + \hat f(-1)) e^y c_{0, \eta}+ \mathrm{O}_{\Gamma, f, \eta_0}\left(\frac{1}{\eta^2} + 1\right),
\end{equation*}
where the sum is over principal series representations of $L^2(\Gamma \backslash G)$, $c_{s, \eta}$ is defined as in \eqref{c-s-eta-def}, and $\tilde L[f,g_{y, \eta}]$ is defined as in \eqref{L-tilde-def}.
\end{prop}

\begin{proof} First, we note that $f = \feven + \fodd$ where $\feven(x) = \frac{f(x) + f(-x)}{2}$ and $\fodd(x) = \frac{f(x) -f(-x)}{2}$. Then we decompose $\tilde L[f, g_{y, \eta}]$ into
 \begin{align*}
\tilde L[f, g_{y, \eta}] &= \sum_{[\gamma]} \ell(\gamma_0) w(\gamma) \tilde g_{y, \eta}(\ell(\gamma))\feven(\hol(\gamma))+ \sum_{[\gamma]}\ell(\gamma_0) w(\gamma) \tilde h_{y, \eta}(\ell(\gamma))  \fodd(\hol(\gamma))+ \mathrm{O}_{\Gamma, f, \eta_0}(1)\\
&= \tilde L_{\text{even}}[f, g_{y, \eta}]  + \tilde L_{\text{odd}}[f, g_{y, \eta}] + \mathrm{O}_{\Gamma, f, \eta_0}(1).
\end{align*}
	
To estimate $\tilde L_{\text{even}}[f, g_{y, \eta}]$, we will use the function $\tilde g_{y, \eta}(u) \feven(\theta)$, where $\tilde g_{y, \eta}(u) = (e^u + e^{-u}) g_{y, \eta}(u)$, in the preliminary trace formula from Proposition \ref{separable-trace}. In the geometric term, we have
\[\tilde L_{\text{even}}[f,g_{y, \eta}]  := \sum_{[\gamma]} \ell(\gamma_0) g_{y, \eta}(\ell(\gamma)) (e^{\ell(\gamma)} + e^{- \ell(\gamma)}) w(\gamma) f_{\text{even}}(\hol( \gamma)),\]
where $w(\gamma) = |1- e^{\mathbb{C}\ell(\gamma)}|^{-1} |1- e^{-\mathbb{C}\ell(\gamma)}|^{-1}$.

By the trace formula, this is equal to:
\begin{align*}
\tilde L_{\text{even}}[f,g_{y, \eta}] &= \frac12 \sum_{\nu,p} (m_\Gamma(\pi_{\nu,p}) + m_\Gamma(\pi_{\nu, -p})) \hat f_{\text{even}}(p) \hat{\tilde  g}_{y, \eta}( \frac{i \nu}{2 \pi})+\frac{1}{2} \hat{f}_{\text{even}}(0) \int_{-\infty}^\infty (e^u + e^{-u}) \tilde g(u) \, du\\
 &- \hat{\tilde g}_{y, \eta}(0) \hat f_{\text{even}}(1)+  \frac{1}{2 \pi} \vol(\Gamma \backslash G)(\tilde g_{y, \eta}''(0) f(0) + \tilde g_{y, \eta}(0) f_{\text{even}}''(0)),
 \end{align*}
where each spectral term is $\gtildehat(i \nu/2 \pi) = \int_{\mathbb{R}} g_{y, \eta}(x) (e^x + e^{-x}) e^{\nu x} \, dx$. Since $\hat f_{\text{even}}(0) = \hat f(0)= 0$, the contributions of the complementary series representations and trivial representation will vanish. The final term is $\mathrm{O}_{\Gamma, f}(1)$.

The remaining spectral terms correspond to principal series representations $\pi_{is, p}$ for $p \neq 0$ and $s \in \mathbb{R}$. We now explicate these terms.
\begin{align*}
\gtildehat(s/2 \pi) &= \int_{-\infty}^\infty g_{y, \eta}(x) (e^x + e^{-x}) e^{isx} \,dx\\
& = \int_{-\infty}^\infty \Big(\int_{-\infty}^\infty \psi_\eta(t) \chi_{[-y,y]}(x-t) \,dt) (e^{x(1 + is)} + e^{x(-1+is)}\Big) \, dx\\
& = \int_{-\infty}^\infty \psi_\eta(t)\Big( \int_{-y +t}^{y +t} e^{x(1 + is)} + e^{x(-1+is)} \,dx \Big)\, dt\\
& = \int \psi_\eta(t) \left( \frac{e^{(y+t)(1+is)} - e^{(-y+t)(1+is)}}{1+is}  + \frac{e^{(y+t)(-1+is)} - e^{(-y+t)(-1+is)}}{-1+is} \right) \, dt\\
& = (e^{y(1+is)} - e^{-y(1+is)}) \left( \int_{-\infty}^\infty \psi_\eta(t) \frac{e^{t(1+is)}}{1+is} \, dt\right) \\
&\quad \quad \quad \quad \quad \quad+ (e^{-y(-1+is)} - e^{y(-1+is)}) \left( \int_{-\infty}^\infty  \psi_\eta(t) \frac{e^{t(1-is)}}{1-is} \, dt \right)\\
& = e^{y}(e^{iys} c_{s, \eta} + e^{-iys}c_{-s, \eta})  - e^{-y}(e^{-iys} c_{s, \eta} + e^{iys} c_{-s, \eta})\\
&=2 e^y \Re(e^{isy}c_{s, \eta}) - 2 e^{-y} \Re(e^{-isy} c_{s, \eta})
\end{align*}
where $s \in \mathbb{R}$ and $c_{-s, \eta} = \overline{ c_{s, \eta}}$.

Similarly, for $s = 0$, we have
\[\gtildehat(0) = 2e^y \int \psi_\eta(t) e^t \, dt -2 e^{-y} \int \psi_\eta(t) e^t \, dt = 2 (e^y-e^{-y}) c_{0, \eta}.\]

We also compute that $\tilde g_{y, \eta}''(0) = \tilde g_{y, \eta}(0) = 2$.

For the odd part, we use the odd trace formula from Proposition \ref{separable-trace} with the function $\tilde h_{y, \eta}(u) \fodd(\theta)$, where $\tilde h_{y, \eta}(u) = (e^u + e^{-u}) h_{y, \eta}(u)$ and $h_{y, \eta}$ is defined in \eqref{g-y-eta-def}, to obtain
\begin{align*}
\tilde L_{\text{odd}}[f, g_{y, \eta}] &= \sum_{[\gamma]} \ell(\gamma_0) w(\gamma)  \tilde h_{y, \eta}(\ell(\gamma)) \fodd (\hol(\gamma))\\
&=\frac12 \sum_{\nu,p} (m_\Gamma(\pi_{\nu,p})- m_\Gamma (\pi_{\nu, -p})) \hat f_{\text{odd}}(-p) \hat{\tilde h}_{y, \eta}( \frac{i \nu}{2 \pi}).
\end{align*}

Note that $\hat f_{\text{odd}}(0) = 0$ and so we are left with principal series representations where $\nu = is$. We calculate that
\begin{align*}
\hat{\tilde h}(-s/2 \pi) &= \int_{-\infty}^\infty h_{y, \eta}(x) (e^x + e^{-x}) e^{isx} \,dx\\
& = \int_{-\infty}^\infty \Big(\int_{-\infty}^\infty \psi_\eta(t) (\chi_{[0,y]}(x-t) - \chi_{[-y,0]}(x-t)\,dt) (e^{x(1 + is)} + e^{x(-1+is)}\Big) \, dx\\
& = \int_{-\infty}^\infty \psi_\eta(t)\Big( \int_{t}^{y +t} e^{x(1 + is)} + e^{x(-1+is)} \,dx -\int_{-y+t}^{t} e^{x(1 + is)} + e^{x(-1+is)} \,dx \Big)\, dt\\
& = \int \psi_\eta(t) \Big( \frac{e^{(y+t)(1+is)} - e^{(-y+t)(1+is)}}{1+is}  + \frac{e^{(y+t)(-1+is)} - e^{(-y+t)(-1+is)}}{-1+is} \Big) \, dt\\
& =(e^{y(1+is)} + e^{-y(1+is)} - 2) c_{s, \eta} - (e^{y(-1+is)} + e^{-y(-1+is)} - 2) c_{-s, \eta}\\
& =2 i e^y \Im(e^{isy} c_{s, \eta}) + \mathrm{O} (|c_{s, \eta}|).
\end{align*}
%\[\hat {\tilde h}_{y, \eta}(\frac{-s}{2\pi}) = 2 i e^y \Im(e^{isy} c_{s, \eta}) + \mathrm{O}_{\Gamma, \eta_0, N}(\frac{1}{(s+1)^N}).\]
Then noting that $\hat f_{\text{even}}(p) = \Re(\hat f(-p)) $ and $\hat f_{\text{odd}}(-p)= i \Im(\hat f(-p))$, we observe that
\[\Re(\hat f(-p)) \Re(e^{isy}c_{s,\eta}) - \Im (\hat f(-p)) \Im (e^{isy} c_{s, \eta}) = \Re( \hat f(-p) e^{isy} c_{s, \eta}),\]
which leads to the estimate 
\[\tilde L[f, g_{y, \eta}] = e^y \sum_{s,p} m_\Gamma(\pi_{is,p})  \Re\big(\hat f(-p) e^{isy} c_{s, \eta}\big)  + \mathrm{O}\bigg( \sum_{s,p} m_\Gamma(\pi_{is,p})) |\hat f(-p)| |c_{s, \eta}|\bigg) + \mathrm{O}_{\Gamma, f, \eta_0}(1).\]
%where the error term is bounded using the bound on multiplicities from Proposition \ref{spectral-bound}, the Schwartz decay of $c_{s, \eta}$ for a fixed $\eta$, and the Schwartz decay of $\hat f(p)$.

To bound the error term, we need  the Schwartz decay of $\hat f(p)$ and the bound on $c_{s, \eta}$  from Remark \ref{cs-decay}. In particular, there exists $n$, $m$ such that $\hat f(p) \ll_{f,n} \frac{1}{|p|^n}$ and $|c_{s,\eta}| \ll_{\eta_0,m} \frac{1}{(|s|+1)(|\eta s|^m + 1)}$. Then, breaking the spectral parameters into length $1$ intervals and  using the  bound on multiplicities from Proposition \ref{spectral-bound},
\begin{align*}
\sum_{s,p} m_\Gamma(\pi_{is,p})) |\hat f(-p)| |c_{s, \eta}| & \ll_{f,\eta_0, n,m} \sum_{s,p \neq 0} m_{\Gamma}(\pi_{is,p}) \frac{1}{|p|^n} \frac{1}{(|s|+1)(|\eta s|^m +1)}\\
& \ll_{ f, \eta_0, n,m} \sum_{p \neq 0} \frac{1}{|p|^n} \sum_{R \in  \mathbb{Z}} \frac{1}{(|R|+1)(|\eta R|^m+1)}\sum_{R \leq |s| < R+1} m_\Gamma(\pi_{is,p})\\
& \ll_{\Gamma, f, \eta_0, n,m} \sum_{p \neq 0} \frac{1}{|p|^n} \sum_{R \in \mathbb{Z}} \frac{R^2 + p^2}{(|R|+1)(|\eta R|^m+1)} .
\end{align*}
 Then we obtain that the error is
 \begin{align*}
 \ll_{\Gamma, f, \eta_0, n,m}& \sum_{p \neq 0} \frac{1}{|p|^n} \bigg(\sum_{0\leq |R| <1} R^2+  \sum_{1 \leq |R|<1/\eta} |R|  + \sum_{|R| \geq 1/\eta} \frac{1}{\eta^m |R|^{m-1}} \bigg)\\
+& \sum_{p \neq 0} \frac{1}{|p|^{n-2}} \bigg( \sum_{0 \leq |R|<1} 1+  \sum_{1 \leq |R|<1/\eta} \frac{1}{|R|}  +  \sum_{|R| \geq 1/\eta} \frac{1}{|\eta R|^m}\bigg).
 \end{align*}
 Then choosing, say, $n = 4$ and $m = 3$ yields an error term  $\ll_{\Gamma, f, \eta_0} \frac{1}{\eta^2}+1$.
\end{proof}

Next, we obtain an estimate on the contribution from the change of weights from the count obtained by the trace formula in Proposition \ref{geom-trace-est} to the desired weight in \eqref{bias-count-def}.

\begin{lemma}\label{change-of-weight} For any co-compact, torsion-free discrete subgroup $\Gamma < \PSL$, $0<\eta \leq  \eta_0 <y$, $g_{y, \eta}$ defined as in \eqref{g-y-eta-def}, and $f: \mathbb{R}/2 \pi \mathbb{Z} \rightarrow \mathbb{R}$ smooth such that $\hat f(0) = 0$,
\begin{multline*}
\sum_{[\gamma]} \ell(\gamma_0)\left(1 - \frac{e^{\ell(\gamma)} + e^{-\ell(\gamma)}}{e^{\ell(\gamma)} + e^{-\ell(\gamma)}- 2 \cos(\hol \gamma)}\right) g_{y, \eta}(\ell(\gamma)) f(\hol \gamma) \\
= -\big( \hat f(1)+ \hat f(-1)\big) (e^y - e^{-y}) c_{0, \eta} + \mathrm{O}_{\Gamma, f, \eta_0}\left(y + \frac{1}{\eta^2} +1\right).
\end{multline*}
\end{lemma}

\begin{proof} First, we calculate that $1 - \frac{e^{\ell(\gamma)} + e^{-\ell(\gamma)}}{e^{\ell(\gamma)} + e^{-\ell(\gamma)} - 2 \cos(\hol(\gamma))} = -2 \cos(\hol(\gamma)) w(\gamma)$. Thus we have 
\begin{align*}
\sum_{[\gamma]} \ell(\gamma_0)&\left(1 - \frac{e^{\ell(\gamma)} + e^{-\ell(\gamma)}}{e^{\ell(\gamma)} + e^{-\ell(\gamma)}- 2 \cos(\hol \gamma)}\right) g_{y, \eta}(\ell(\gamma)) f(\hol \gamma)\\
 &= -2 \sum \ell(\gamma_0) w(\gamma) g_{y, \eta}(\ell(\gamma)) \cos(\hol \gamma) f(\hol \gamma).
\end{align*}
We again use the decomposition $f(\theta) = \feven(\theta) +\fodd(\theta)$ to achieve
\begin{align*}
-2 \sum \ell(\gamma_0)& w(\gamma) g_{y, \eta}(\ell(\gamma)) \cos(\hol \gamma) f(\hol \gamma)\\
 &= -2 \sum \ell(\gamma_0) w(\gamma) g_{y, \eta}(\ell(\gamma)) \cos(\hol \gamma) \feven(\hol \gamma) \\
&-2 \sum \ell(\gamma_0) w(\gamma) h_{y, \eta}(\ell(\gamma)) \cos(\hol \gamma) \fodd(\hol \gamma)+ \mathrm{O}_{\Gamma, f,\eta_0}(1),
\end{align*}
where the first and second sums on the right hand side can be estimated by the even and odd trace formulas, respectively, of Proposition \ref{separable-trace}, and the error term accounts for geodesics with length less than $\eta_0$. For the first, we use the even function $\tilde f_{\text{even}} (\theta)= \cos(\theta) \feven (\theta)$ on the holonomy, and the even function $g_{y, \eta}$ on the length. Then the trace formula gives us the equality:
\begin{align*}
-2 \sum \ell(\gamma_0) &w(\gamma) g_{y, \eta}(\ell(\gamma)) \cos(\hol \gamma) f_{\text{even}} (\hol \gamma)\\
&=- \sum_{\nu,p} \left(m_\Gamma(\pi_{\nu,p})+ m_\Gamma(\pi_{\nu, -p})\right) \hat{\tilde f}_{\text{even}} (p) \hat g_{y, \eta}(\frac{i \nu}{2 \pi}) - \hat{\tilde f}_{\text{even}} (0) \int_{-\infty}^\infty (e^u + e^{-u}) g_{y, \eta}(u) \, du\\
&\quad+ 2 \hat g_{y, \eta}(0) \hat{\tilde f}_{\text{even}} (1) -\frac{1}{\pi} \vol(\Gamma \backslash G)\big(g_{y, \eta}''(0) \tilde f_{\text{even}} (0) +  g_{y, \eta}(0) \tilde f_{\text{even}} ''(0)\big).
\end{align*}

Similarly, we use the odd function $h_{y, \eta}$ defined in \eqref{g-y-eta-def} on the length and the odd function $\tilde f_{\text{odd}}(\theta) = \cos(\theta) \fodd(\theta)$ on the holonomy in the odd trace formula from Proposition \ref{separable-trace} to obtain
\begin{equation*}
-2 \sum \ell(\gamma_0) w(\gamma) h_{y, \eta}(\ell(\gamma)) \cos(\hol \gamma) \fodd(\hol \gamma)= - \sum_{\nu,p} (m_\Gamma(\pi_{\nu,p})- m_\Gamma (\pi_{\nu, -p})) \hat{\tilde f}_{\text{odd}}(-p) \hat{\tilde h}_{y, \eta}( \frac{i \nu}{2 \pi}). 
\end{equation*}

Note that since $\hat f(0) = 0$, we are left with principal series representations where $\nu = is$. Since $\feven$ and $\fodd$ are smooth functions, their Fourier transforms satisfy the Schwartz decay
\begin{equation}\label{hat-tilde-f-bound}
\hat{\tilde f}_{\text{even}}(p)  \ll_{f, n} \frac{1}{1 + |p|^n}, \quad \hat{\tilde f}_{\text{odd}}(p)  \ll_{f, n} \frac{1}{1 + |p|^n}.
\end{equation}

For the spectral terms, we use that 
\[\hat g_{y, \eta}(\frac{-s }{2 \pi}) = 2 \hat \psi(\frac{-\eta s}{2\pi}) \frac{\sin(sy)}{s}, \]
\[\hat h_{y, \eta}(\frac{-s}{2\pi}) = 2 \hat \psi(\frac{-\eta s}{2\pi}) \frac{\cos(sy)-1}{is}.\]

Then, together with the Schwartz bound $|\hat \psi(-\eta s/2 \pi)| \ll_m 1/(1 + \eta|s|)^m$ with, for example, $m=3$, the bound on $\hat{\tilde f}$ from \eqref{hat-tilde-f-bound}, and the bound $\frac{\sin(sy)}{s} \ll \min(y, \frac{1}{s})$, the contribution of the spectral terms is
\begin{align*}
 &\sum_{s,p} \big(m_\Gamma(\pi_{is,p})+ m_\Gamma(\pi_{is, -p})\big) \hat{\tilde f}_{\text{even}} (p) \hat g_{y, \eta}(\frac{-s}{2 \pi}) \\
 &= (\sum_{\substack{0 \leq |s| \leq 1 \\ p \in \mathbb{Z}}} + \sum_{\substack{1< |s| <1/\eta \\ p \in \mathbb{Z}}} + \sum_{\substack{1/\eta<|s| \\ p \in \mathbb{Z}}}) \big(m_\Gamma(\pi_{is,p})+ m_\Gamma(\pi_{is, -p})\big) \hat{\tilde f}_{\text{even}} (p) 2 \hat \psi(\frac{-\eta s}{2\pi}) \frac{\sin(sy)}{s}\\
 & \ll_{f,n,m} y \sum_{\substack{  0 \leq |s| \leq 1 \\p \neq 0}} m_\Gamma(\pi_{is, p}) \frac{1}{ |p|^n} +  \sum_{\substack{1< |s| <1/\eta \\ p \neq 0}} m_\Gamma(\pi_{is, p}) \frac{1}{|p|^n} \frac{1}{s} + \sum_{\substack{1/\eta<|s| \\ p \neq 0}} m_\Gamma(\pi_{is, p}) \frac{1}{|p|^n} \frac{1}{\eta^m |s|^{m+1}}.
 \end{align*}
 Similarly, noting that $\frac{\cos(sy)-1}{is}\ll \min(s y^2, \frac{1}{s}) \ll \min(y, \frac{1}{s})$, we obtain identical bounds for the odd spectral terms.

Then using the spectral bound from Proposition \ref{spectral-bound} on the multiplicities and breaking the sum into length $1$ intervals, we calculate that the above is
\begin{align*}
& \ll_{\Gamma, f, n, m} y \sum_{\substack{0\leq |s| <1 \\ p \neq 0}}  m_\Gamma(\pi_{is, p}) \frac{1}{|p|^n}  + \sum_{p \neq 0} \sum_{1 \leq k < 1/\eta} \frac{1}{k} \sum_{k \leq |s| < k+1} m_\Gamma(\pi_{is, p}) \frac{1}{|p|^n}\\
 &\quad+ \frac{1}{\eta^m}\sum_{p \neq 0} \sum_{k> 1/\eta} \frac{1}{k^{m+1}} \sum_{k \leq |s| \leq k +1} m_\Gamma(\pi_{is, p})\frac{1}{ |p|^n} \\
& \ll_{\Gamma,f, n, m} y \sum_{p \neq 0}\frac{1 + p^2}{|p|^n}  +\sum_{p \neq 0} \sum_{1 < k < 1/\eta} \frac{1}{k} \frac{k^2 + p^2}{|p|^n} +  \frac{1}{\eta^m} \sum_{p\neq 0}\sum_{k> 1/\eta} \frac{1}{k^{m+1}} \frac{k^2 + p^2}{|p|^n} 
\end{align*}
Choosing, say, $n = 4$ and $m = 3$, the spectral terms contribute $\ll_{\Gamma,f} y + \frac{1}{\eta^2} + 1$.

Then, we have that the contribution of the trivial representation is 
\begin{align*}
- \hat f_{\text{even}} (1) \int_{-\infty}^\infty (e^u + e^{-u}) g_{y, \eta}(u) \, du &= - \hat f_{\text{even}} (1) \int_{-\infty}^\infty (e^u + e^{-u}) \left( \int_{- \infty}^\infty \chi_{[-y, y]}(u-x) \psi_\eta(x) \, dx \right) \, du\\
&= - \hat f_{\text{even}} (1)\int_{-\infty}^\infty \psi_\eta(x) \left(\int_{x-y}^{x+y} e^u + e^{-u} \, du\right ) \, dx\\
&= - 2  \hat f_{\text{even}} (1) (e^y - e^{-y}) \int_{-\infty}^\infty \psi_\eta(x) e^x \, dx\\
&= -2 \hat f_{\text{even}} (1) (e^y - e^{-y}) c_{0, \eta}.
\end{align*}

Finally, the remaining terms are $\ll_{ \Gamma,f}y + \eta$.
\end{proof}
Next, we obtain an upper bound that will be useful for bounding the contribution of non-primitive geodesics in Proposition \ref{bias-count}.

\begin{lemma}\label{non-primitive-bound} For a fixed $k \geq 2$ and $0< \eta< \eta_0$, discrete, co-compact, torsion-free $\Gamma <\PSL$, $g_{y, \eta}$ defined as in \eqref{g-y-eta-def}, and $f: \mathbb{R}/2 \pi \mathbb{Z} \rightarrow \mathbb{R}$ smooth with $\hat f(0) =0$,
\[\sump_{[\gamma]} \ell(\gamma) g_{y, \eta}(k \ell(\gamma)) f(k \hol(\gamma)) \ll_{\Gamma, f, k,\eta_0} e^{y/k} \left(1 + \frac{1}{\eta^2}\right),\]
where the sum is over primitive geodesics.
\end{lemma}

\begin{proof}
We will estimate this sum using the trace formula. First we add in non-primitive geodesics, noting that for  $k \geq 2$,
\begin{align*}
\sum^{n.p.}_{[\gamma]} \ell(\gamma_0) g_{y, \eta} (k \ell(\gamma)) f(k \hol(\gamma)) &= \sum_{n \geq 2} \sump_{[\gamma]} \ell(\gamma) g_{y, \eta}(k n \ell(\gamma)) f(k n \hol(\gamma))\\
&\ll_\Gamma \sum_{n \geq 2}\quad \sump_{\ell(\gamma) \leq (y+\eta)/(nk)} \ell(\gamma)\\
& \ll_\Gamma \sum_{n \geq 2} \frac{y+\eta}{nk} \sump_{\ell(\gamma) \leq (y+\eta)/(nk)} 1 \\
& \ll_\Gamma \sum_{n \geq 2} e^{ 2(y+\eta)/(nk)} \ll_\Gamma e^{(y+\eta)/k}
\end{align*}
 by the Prime Geodesic Theorem. Then, we have that
\begin{align*}
\sump_{[\gamma]} \ell(\gamma) g_{y, \eta}(k \ell(\gamma)) f(k \hol(\gamma)) &= \sum_{[\gamma]} \ell(\gamma_0) g_{y, \eta}(k \ell(\gamma)) f(k \hol(\gamma)) + \mathrm{O}_\Gamma( e^{(y+\eta)/k})\\
& =\sum_{[\gamma]} \ell(\gamma_0) g_{y/k, \eta/k}( \ell(\gamma)) f(k \hol(\gamma)) + \mathrm{O}_\Gamma( e^{(y+\eta)/k}).
\end{align*}
Using that $w(\gamma)(e^{\ell(\gamma)} + e^{-\ell(\gamma)})= 1 + \mathrm{O}_\Gamma(e^{-\ell(\gamma)})$, we have that the first term equals
\[\sum_{[\gamma]} \ell(\gamma_0) w(\gamma) \tilde g_{y/k, \eta/k}(\ell(\gamma)) f(k \hol(\gamma)) + \mathrm{O}_\Gamma\bigg( \sum_{\ell(\gamma) \leq (y+\eta)/k} \ell(\gamma) e^{ - \ell(\gamma)}\bigg), \]
and the error term is $\ll_\Gamma e^{(y+\eta)/k}$ by the Prime Geodesic Theorem. We denote $f_k(\theta) = f(k \theta)$ and note that if $\hat f(0) = 0$, then $\hat f_k(0) = 0$. Then by Proposition \ref{geom-trace-est}, 
\begin{align*}
\sum_{[\gamma]}& \ell(\gamma_0) w(\gamma) \tilde g_{y/k, \eta/k}(\ell(\gamma)) f(k \hol(\gamma)) \\
&= 2 e^{y/k} \sum_{s, p} m_{\Gamma}(\pi_{is,p}) \Re( \hat f_k(-p)e^{isy} c_{s, \eta}) - (\hat f_k(1) + \hat f_k(-1)) e^{y/k} c_{0, \eta} + \mathrm{O}_{\Gamma, f, k, \eta_0}\left(\frac{1}{\eta^2} + 1\right) \\
& \ll_{\Gamma,k, \eta_0} e^{y/k}\bigg(\sum_{s,p} m_\Gamma(\pi_{is,p})  
|\hat f_k(-p)| |c_{s, \eta}| + |\hat f_k(1) + \hat f_k(-1)||c_{0, \eta}| \bigg) + \mathrm{O}_{\Gamma, f, k, \eta_0}\left(\frac{1}{\eta^2} + 1\right).
\end{align*}

Remark \ref{cs-decay} gives the bound $|c_{s, \eta}| \ll_{n,\eta_0} \frac{1}{(|s|+1)(|\eta s|^n+1)} $, and Schwartz decay of Fourier coefficients gives $\hat f_k(p) \ll_{m,f,k} \frac{1}{1+ |p|^m}$. Using the spectral bound from Proposition \ref{spectral-bound} on multiplicities,
\begin{align*}
&\sum_{s,p} m_\Gamma(\pi_{is,p}) \hat f_k(-p) |c_{s, \eta}| = \sum_{p} \bigg(\sum_{|s|<1/\eta} + \sum_{|s|\geq 1/\eta}\bigg)m_\Gamma(\pi_{is,p}) \hat f_k(-p) |c_{s, \eta}|\\
&\ll_{\Gamma,f,k, n,m, \eta_0}  \sum_{p \neq 0} \frac{1}{|p|^m}\sum_{0 \leq |s|< \frac{1}{\eta}} m_\Gamma(\pi_{is, p} ) \frac{1}{|s|+1}  +  \sum_{p \neq 0} \frac{1}{|p|^m}\sum_{|s|\geq \frac{1}{\eta}} m_\Gamma(\pi_{is, p} )\frac{1}{|s|^{n+1}\eta^n}\\
& \ll_{\Gamma,f,k, n,m, \eta_0} \sum_{p \neq 0} \frac{1}{|p|^m} \sum_{0 \leq R<\frac{1}{\eta}}  \frac{R^2 + p^2}{R+1} + \frac{1}{\eta^n} \sum_{p \neq 0} \frac{1}{|p|^m} \sum_{R\geq\frac{1}{\eta}}  \frac{R^2 + p^2}{R^{n+1}}
\end{align*}
where the inner sums are over integers. Choosing, say, $n = 3$ and $m=4$ results in  $\ll_{\Gamma,f, k, \eta_0} \frac{1}{\eta^2}+1$. 
\end{proof}

Finally, we are able to obtain an estimate for the bias count $L^P[f, g_{y, \eta}]$.

\begin{prop}\label{bias-count} Let $\Gamma < \PSL$ be a discrete, co-compact, torsion-free subgroup, $g_{y, \eta}$ be defined as in \eqref{g-y-eta-def}, $0<\eta \leq \eta_0< y$, and $f: \mathbb{R}/2 \pi \mathbb{Z} \rightarrow \mathbb{R}$ be a smooth function with $\hat f(0)=0$. Then
\begin{align*}
&L^P[f, g_{y, \eta}] := \sump_{[\gamma]} \ell(\gamma) g_{y, \eta}(\ell(\gamma))f(\hol(\gamma))\\
&= 2 e^y \sum_{s,p} m_\Gamma(\pi_{is,p})  \Re\left( \hat f(-p)e^{isy}c_{s, \eta}\right) - 2 e^y (\hat f(1) + \hat f(-1)) c_{0, \eta} + \mathrm{O}_{\Gamma, f, \eta_0}\left(e^{y/2}\left(1 + \frac{1}{\eta^2}\right)\right)\\
& = 2 e^y \sum_{\substack{s\neq 0, p}} m_\Gamma(\pi_{is,p}) \Re\left( \hat f(-p) e^{isy}c_{s, \eta}\right)  + e^y b_{f,\eta}  + \mathrm{O}_{\Gamma, f, \eta_0}\left(e^{y/2}\left(1 + \frac{1}{\eta^2}\right)\right)
\end{align*}
where $b_{f,\eta}$ is defined as in \eqref{b-f-eta-background}, $c_{s, \eta}$ is defined as in \eqref{c-s-eta-def}, and the sum is over the principal series representations $\pi_{is, p}$ in $L^2(\Gamma \backslash G)$.

\end{prop}

\begin{proof} To estimate $L^P[f, g_{y, \eta}]$ with the trace formula, we will also need to account for the addition of non-primitive geodesics and the removal of weights. Formally, this is accounted for by the equation
\[L^P[f, g_{y, \eta}] = (L^P[f, g_{y, \eta}] -  L[f,g_{y, \eta}]) + (L[f,g_{y, \eta}] - \tilde L[f, g_{y, \eta}]) + \tilde L[f,g_{y, \eta}],\]
where $\tilde L[f, g_{y, \eta}]$ is defined in \eqref{L-tilde-def} and $L[f, g_{y, \eta}]= \sum_{[\gamma]} \ell(\gamma) g_{y, \eta}(\ell(\gamma)) f(\hol(\gamma))$.

The contribution of the non-primitive geodesics is
\begin{align*}
- \sum_{[\gamma] \text{ n.p.}}  \ell(\gamma) g_{y, \eta}(\ell(\gamma)) f(\hol(\gamma))  &= - \sump_{[\gamma]} \sum_{k \geq 2} \ell(\gamma) g_{y, \eta}(\ell(\gamma^k)) f(\hol(\gamma^k))\\
&= -  \sum_{k \geq 2} \sump_{[\gamma]}\ell(\gamma) g_{y, \eta}(k \ell(\gamma)) f(k \hol(\gamma)).
\end{align*}
The $k = 2$ and $3$ terms are $\ll_{\Gamma,f, \eta_0} e^{y/2} (1 + \frac{1}{\eta^2})$ by Lemma \ref{non-primitive-bound}. For $k>3$, we use the Prime Geodesic Theorem to obtain the bound
\[\sump_{[\gamma]}\ell(\gamma) g_{y, \eta}(k \ell(\gamma)) f(k \hol(\gamma)) \ll \|f\|_\infty  \sump_{\ell(\gamma) \leq (y + \eta)/k} \ell(\gamma) \ll_{\Gamma} \|f\|_\infty e^{2(y+\eta)/k}\]
which is uniform in $k$. Summing over all $k>3$ yields $\ll_{\Gamma,f} e^{ (y+\eta)/2}$.

By Lemma \ref{change-of-weight}, the contribution from the change of weights is
\[L[f,g_{y, \eta}] - \tilde L[f, g_{y, \eta}] =-(\hat f(1) + \hat f(-1)) (e^y - e^{-y}) c_{0, \eta} + \mathrm{O}_{\Gamma, f, \eta_0}\Big(y + \frac{1}{\eta^2}\Big).\]
Proposition \ref{geom-trace-est} gives us the trace formula estimate for $\tilde L[f,g_{y, \eta}]$, which completes the proof.
\end{proof}

\begin{remark} In the case of $f(\theta) = \cos(n \theta)$, we then have that $\hat f(n) = \hat f(-n)= \frac{1}{2}$ and all other Fourier coefficients vanish. In that case, Proposition \ref{bias-count} simplifies to 
\[L^P[f_n, g_{y, \eta}] = e^y \frac{1}{2} \sum_{\nu \neq 0} (m_\Gamma(\pi_{\nu,n}) + m_\Gamma(\pi_{\nu, -n})) \Re(e^{isy}c_{s, \eta}) + e^y b_{n,\eta}  + \mathrm{O}_{\Gamma, \eta_0}\bigg(e^{y/2}\Big(1 + \frac{1}{\eta^2}\Big)\bigg) \]
where $b_{n,\eta} = (\big( \# (\nu, p) = (0,n) \big)- 2 \delta_1(n)) c_{0, \eta}$ is the bias constant. If there are no zero spectral parameters, then a bias only exists for $n =1$, in which case there is a bias in the negative direction.

In general, there is no bias exactly when $\sum_{p \neq 0} m_\Gamma(\pi_{0, p})\Re(\hat f(p)) = 2 \Re(\hat f(1))$; for example, when $m_\Gamma(\pi_{0, 1}) = 2$ and $m_\Gamma(\pi_{0, p}) = 0$ for all $|p|>1$. For positive $f$, we would typically expect the bias constant $b_{f, \eta}$ to be negative, although a large number of zero spectral parameters could potentially lead to a positive bias constant.

\end{remark}

\section{Distribution of Bias in Holonomy}\label{Bias-Distribution-Section}

In the last section, we established a smoothed count of the bias,
\[L^P[f,g_{ y, \eta}] = \sump_{[\gamma]} \ell(\gamma) g_{y, \eta}(\ell(\gamma)) f(\hol(\gamma)).\]
In Proposition \ref{bias-count}, we were able to obtain an asymptotic formula, in which the main term is a rapidly convergent  linear combination of terms of size $\asymp e^y$ as $y \rightarrow \infty$. In this section, we define the normalized count of bias by dividing by $e^y$:
\begin{equation}\label{E-f-g-def}
\begin{aligned}
E[f, g_{y, \eta}] &:= e^{-y} L^P[f,g_{y,\eta}]\\
&= 2 \sum_{\substack{s \neq 0\\ p\neq 0}} m_\Gamma(\pi_{is,p})  \Re\left(\hat f(-p) e^{isy}c_{s, \eta}\right)  + b_{f, \eta} + \mathrm{O}_{\Gamma, f, \eta_0}\left(e^{-y/2}\left(1 + \frac{1}{\eta^2}\right)\right),
\end{aligned}
\end{equation}
where $b_{f,\eta} = \big(\sum_{p \neq 0} m_\Gamma(\pi_{0, p}) \Re(\hat f(p))  -  2 \Re(\hat f(1))\big) 2 c_{0, \eta}$ and $0<\eta\leq \eta_0$. Note that we have the  bound
\begin{equation}\label{E-unif-bound}
E[f, g_{y, \eta}] \ll \sum_{\substack{s \neq 0\\ p\neq 0}} m_\Gamma(\pi_{is,p})  |\hat f(-p)| |c_{s, \eta}|  + |b_{f, \eta}| + \mathrm{O}_{\Gamma, f, \eta}(1)
\end{equation}
which is uniform in $y$, so $E[f, g_{y, \eta}]$ is compactly supported.

Ultimately, we will show that the values of $E[f, g_{y,\eta}]$ converge  to a probability distribution as $y \rightarrow \infty$. When the spectral parameters of $L^2(\Gamma \backslash G)$ are linearly independent, we will be able to establish an explicit distribution and show that this is symmetric about $b_{f,\eta}$. For any set of spectral parameters, we can establish that the average value of $E[f,g_{y, \eta}]$ is $b_{f, \eta}$, although the distribution may be asymmetric.

To estimate $E[f,g_{y, \eta}]$, we  start with the finite spectral cutoff
\begin{equation}\label{ETdef}
E^{(T)}[f, g_{y, \eta}] := 2 \sum_{\substack{|s|, |p|<T\\ s \neq 0}} m_\Gamma(\pi_{is,p}) \Re\left(\hat f(-p) e^{isy}c_{s, \eta}\right) +b_{f, \eta},
\end{equation}
noting that we are assuming $\hat f(0)= 0$.  We also note that due to the spectral bound from Prop. \ref{spectral-bound} only finitely many terms are non-zero; we index the spectral parameters $s_1$,..., $s_n$. In addition, each of these sums is compactly supported in a fixed interval since
\begin{equation}\label{ET-bound}
|E^{(T)}[f, g_{y, \eta}]| \leq 2 \sum_{s, p} m_{\Gamma}(\pi_{is,p}) |\hat f(-p)| |c_{s, \eta}| + |b_{f, \eta}|.
\end{equation}

We then use the Kronecker-Weyl theorem  to show that $E^{(T)}[f,g_{y, \eta}]$ is distributed according to a probability distribution, $\mu_{T, f, \eta}$. Given a continuous function $h: \mathbb{R}\rightarrow \mathbb{C}$, we define the function
\[h_{T,f,\eta}(x_1,..,x_n) = h\left(2 \sum_{j =1}^n m_\Gamma(\pi_{is_j,p_j})  \Re\left(\hat f(-p_j) e^{2 \pi i x_j}c_{s_j, \eta}\right) +b_{f, \eta}\right)\]
on $\mathbb{R}^n/\mathbb{Z}^n = \mathbb{T}^n$.

\begin{lemma}\label{ETLemma} Let $\Gamma < \PSL$ be a discrete, co-compact, torsion-free subgroup, and let $h:\mathbb{R} \rightarrow \mathbb{C}$ be continuous.  Suppose that $L^2(\Gamma \backslash G)$ has non-zero spectral parameters $s_1, ..., s_n$ with $|s_j|, |p_j|<T$. For a fixed $\eta>0$, as $y \rightarrow \infty$, $E^{(T)}[f, g_{y, \eta}]$ has the probability distribution
\[\lim_{Y \rightarrow \infty} \frac1Y \int_{\eta_0}^Y h(E^{(T)}[f, g_{y, \eta}])\, dy = \int_A h_{T,f,\eta}(a) \,da = \int_{\mathbb{R}} h(x) \,d\mu_{T, f, \eta}(x)\]
where $h_{T,f,\eta}$ is as above and $A$ is the closure of $\{(\frac{s_1 y}{2\pi},...,\frac{s_n y}{2\pi}), y \in \mathbb{R}\}$ in $\mathbb{T}^n$, which is a subtorus of dimension $r$. $E^{(T)}[f,g_{y, \eta}]$ is defined as in \eqref{ETdef}, $b_{f,\eta}$ as in \eqref{b-f-eta-background}, and $c_{s, \eta}$  as in \eqref{c-s-eta-def}.
\end{lemma}

\begin{proof} First we observe that
\[h(E^{(T)}[f,g_{y, \eta}]) = h_{T,f, \eta}\left(\frac{s_1 y}{2 \pi}, ..., \frac{s_n y}{2 \pi}\right),\]
where $h_{T,f, \eta}$ is defined above. Note that $h_{T,f, \eta}$ is a continuous function on the torus $\mathbb{T}^n$.

 Then by the Kronecker-Weyl Theorem,
\[\lim_{Y \rightarrow \infty} \frac1Y \int_{\eta_0}^Y h(E^{(T)}[f, g_{y, \eta}])\, dy = \int_A h_{T,f,\eta}(a) \,da,\]
where $A$ is the subtorus of $\mathbb{T}^n$ defined above and $da = d\mu_A(a)$, where $\mu_A$ is the Haar measure on $A$.

Recall that 
 \[h_{T,f,\eta}(x_1,..,x_n) = h\left(2 \sum_{j =1}^n m_\Gamma(\pi_{is_j,p_j})  \Re\left(\hat f(-p) e^{2 \pi i x_j}c_{s_j, \eta}\right) +b_{f, \eta}\right).\]
 Note that $w_{T, f, \eta}(x_1,..., x_n) := 2 \sum_{j =1}^n m_\Gamma(\pi_{is_j,p_j})  \Re(\hat f(-p) e^{2 \pi i x_j}c_{s_j, \eta}) +b_{f, \eta}$ is a real-valued function on $A$ and so, by a standard change of variables formula (see, for example \cite[Theorem 16.12]{Bill})), there exists a measure $\mu_{T, f, \eta}$ on $\mathbb{R}$ defined by 
 \[\mu_{T, f, \eta}(B) = \mu_A\left(w_{T, f, \eta}^{-1}(B)\right) \text{ for all } B \subseteq \mathbb{R} \text{ measurable},\]
 so that 
\[ \int_A h_{T,f, \eta}(a) \, da  = \int_{\mathbb{R}} h(x) \, d\mu_{T, f, \eta}(x).\]
\end{proof}

\begin{cor}\label{ET-mean} For any  $T>0$, $\eta>0$, $\mu_{T, f, \eta}$, and $\Gamma$ as in Lemma~\ref{ETLemma},

1) The distribution $\mu_{T, f, \eta}$ has mean $b_{f, \eta}$.

2) The distribution $\mu_{T, f,  \eta}$ is compactly supported on the interval 
\begin{multline*}
[b_{f, \eta} - 2 \sum_{|s|, |p|<T} m_\Gamma(\pi_{is,p})  |\hat f(-p)| |c_{s, \eta}|, b_{f, \eta} + 2 \sum_{|s|, |p|<T} m_\Gamma(\pi_{is,p})  |\hat f(-p)| |c_{s, \eta}|]\\
 \subset [b_{f, \eta} - 2\sum_{s, p} m_\Gamma(\pi_{is,p})  |\hat f(-p)| |c_{s, \eta}|,b_{f, \eta} + 2 \sum_{s,p} m_\Gamma(\pi_{is,p}) |\hat f(-p)| |c_{s, \eta}|].
\end{multline*}

\end{cor}

\begin{proof}
The mean of the distribution is obtained by considering the function $h(x) = x$ in Lemma \ref{ETLemma}. The compact support follows immediately from considering that the distribution is of \allowbreak$\sum_{j=1}^n m_\Gamma(\pi_{is_j,p_j}) \Re(\hat f(-p_j) e^{2 \pi i x_j}c_{s_j, \eta}) +b_{f, \eta}$ which lies in this interval.
\end{proof}

If the spectral parameters $s_{1}, ..., s_n$ are linearly independent, then we are able to obtain a more explicit distribution in Lemma \ref{ETLemma}.

\begin{cor}\label{ET-measure-LI} If $s_1$, ..., $s_n$ are linearly independent over $\mathbb{Q}$, then $A = \mathbb{T}^n$ and for a fixed $\eta>0$,
\begin{align*}
\lim_{Y \rightarrow \infty}& \frac1Y \int_{\eta_0}^Y h(E^{(T)}[f, g_{y, \eta}]) \,dy \\
&= \int_0^1\cdots\int_0^1   h\Big(2 \sum_{j=1}^n m_\Gamma(\pi_{is_j,p_j})  \Re(\hat f(-p_j) e^{2 \pi i x_j}c_{s_j, \eta}) +b_{f, \eta})\Big) \,dx_1\cdots dx_n.
\end{align*}
\end{cor}

Regardless of linear dependence, we can establish a probability distribution for the normalized bias count, $E[f, g_{y, \eta}]$.

\begin{theorem}\label{Eydistribution} For every $\eta>0$, as $y \rightarrow \infty$, $E[f,g_{y,\eta}]$ has a probability distribution $\mu_{f,\eta}$ such that
\[\lim_{Y \rightarrow \infty} \frac1Y \int_{\eta_0}^Y h(E[f,g_{y, \eta}]) \,dy = \int_{\mathbb{R}} h(x) \,d \mu_{f,\eta}(x) = \mu_{f,\eta}(h)\]
for all Lipschitz functions $h:\mathbb{R} \rightarrow \mathbb{C}$. Further, the distribution $\mu_{f,\eta}(h) = \mu_{T, f,\eta}(h) + \mathrm{O}_{\Gamma, f,\eta,N}\left(\frac{k_h}{T^N}\right)$ where $k_h$ is the Lipschitz constant of $h$ and $\mu_{T, f,\eta}(h)$ is as in Lemma \ref{ETLemma}.
\end{theorem}

\begin{proof}
First, let 
\begin{align*}
\epsilon(f, y, \eta, T) &:= E[f,g_{y, \eta}] - E^{(T)}[f,g_{y, \eta}] \\
&=  2 \sum_{|s|\text{or} |p|>T} m_\Gamma(\pi_{is,p})  \Re\left(\hat f(-p) e^{isy}c_{s, \eta}\right)  + \mathrm{O}_{\Gamma, f, \eta_0}\left(e^{-y/2}\left(1 + \frac{1}{\eta^2}\right)\right).
\end{align*}

 Since $h$ is Lipschitz,
\[\frac1Y \int_{\eta_0}^Y h(E[f,g_{y, \eta}]) \, dy = \frac1Y \int_{\eta_0}^Y h(E^{(T)}[f, g_{y, \eta}]) \,dy + \mathrm{O}\left(k_h  \frac1Y \int_{\eta_0}^Y |\epsilon(f, y, \eta, T)| \, dy\right).\]

For the error term, we have the upper bound
\begin{equation*}
k_f \frac1Y \int_{\eta_0}^Y |\epsilon(f,y,\eta,T)| \, dy  \ll_{\Gamma, f, \eta}  k_h \bigg(\sum_{|s|>T} \sum_p + \sum_s \sum_{p>T}\bigg) m_\Gamma(\pi_{is,p}) |\hat f(-p)| |c_{s, \eta}| +  \frac{k_h}{Y}.
\end{equation*}

For the spectral terms, we  use the Schwartz decay $\hat f(p) \ll_N \frac{1}{|p|^N +1}$ and $c_{s, \eta} \ll_N \frac{1}{|s|^N+1}$, along with the spectral bound from Prop. \ref{spectral-bound}, to obtain that the error term is
\begin{align*}
&\ll_{\Gamma, f, \eta, N} k_h \left(\sum_{n>T} \sum_p + \sum_n \sum_{p>T} \right) \sum_{\substack{n<|s|<n+1}} m_\Gamma(\pi_{is,p})   \frac{1}{|p|^N+1} \frac{1}{|s|^N+1}\\
&\ll_{\Gamma, f, \eta, N} k_h \left(\sum_{n>T} \sum_{p<T}  \frac{n^2+ p^2}{n^N}+ \sum_{n<T}\sum_{p>T}\frac{n^2 + p^2}{p^N} + \sum_{n>T} \sum_{p>T}\frac{n^2+p^2}{n^N p^N}\right)\\
&\ll_{\Gamma, f, \eta, N} \frac{k_h}{T^{N-4}}.
\end{align*}

Let $N' = N-4$. This gives us the bound
\[ \frac1Y \int_{\eta_0}^Y h(E[f,g_{y, \eta}]) \, dy = \frac1Y \int_{\eta_0}^{Y} h(E^{(T)}[f, g_{y, \eta}]) \, dy + \mathrm{O}_{\Gamma,f,\eta,N'} \left( k_h \left(\frac{1}{T^{N'}} + \frac{1}{Y}\right) \right).\]

By Lemma \ref{ETLemma}, we know that the limit of the right hand side  (with respect to $Y$) exists. Therefore,
\begin{align*}
&\lim_{Y \rightarrow \infty} \frac1Y \int_{\eta_0}^{Y} h(E^{(T)}[f, g_{y, \eta}]) \, dy - \mathrm{O}_{\Gamma,f,\eta,N'} \left(\frac{ k_h}{T^{N'}} \right)\\ 
&\leq \liminf \frac1Y \int_{\eta_0}^Y h(E[f,g_{y, \eta}]) \, dy\leq \limsup\frac1Y \int_{\eta_0}^Y h(E[f,g_{y, \eta}]) \, dy\\
&\leq \lim_{Y \rightarrow \infty} \frac1Y \int_{\eta_0}^{Y} h(E^{(T)}[f,g_{y, \eta}]) \, dy +\mathrm{O}_{\Gamma,f,\eta,N'} \left( \frac{ k_h}{T^{N'}}  \right). 
 \end{align*}
 By taking $T$ to be arbitrarily large, we can show that $\lim\limits_{Y \rightarrow \infty} \frac1Y \int_{\eta_0}^Y h(E[f, g_{y, \eta}]) \, dy$ exists  and that
 \begin{equation}\label{mu-limit-eq}
  \lim_{Y \rightarrow \infty} \frac1Y \int_{\eta_0}^Y h(E[f,g_{y, \eta}]) \, dy = \mu_{T,f, \eta}(h) + \mathrm{O}_{\Gamma,f, \eta, N'}\left(\frac{k_h}{T^{N'}}\right).
\end{equation}

We can then formally define
\[\mu_{f,\eta}(h) := \lim_{Y \rightarrow \infty} \frac1Y \int_{\eta_0}^Y h\left(E[f,g_{y, \eta}]\right) \, dy.\]
We are not yet claiming that $\mu_{f, \eta}$ is a measure (although this will be shown). By \eqref{mu-limit-eq},  $\mu_{f, \eta}$ satisfies
\begin{equation}\label{mu-T-convergence}
\mu_{f,\eta}(h) = \mu_{T,f, \eta}(h) + \mathrm{O}_{\Gamma,f, \eta, N'}\left(\frac{k_h}{T^{N'}}\right).
\end{equation}

Now we will show that $\mu_{f, \eta}$ is a measure using a strengthened version of L\'evy's continuity theorem (Theorem \ref{Levy}). Observe that
\begin{align*}
\phi_{f, \eta}(\xi) := \int_{\mathbb{R}} e^{-i \xi x} \mu_{f, \eta}(x)  = \mu_{f, \eta}(h_\xi) &= \mu_{T, f, \eta}(h_\xi) + \mathrm{O}_{\Gamma,f, \eta, N'}\left(\frac{k_{\xi}}{T^{N'}}\right)\\
& = \hat \mu_{T, f, \eta}(\xi) + \mathrm{O}_{\Gamma,f, \eta, N'}\left(\frac{k_{\xi}}{T^{N'}}\right)
\end{align*}
for $h_{\xi}(x) = e^{-i \xi x}$ and $k_{\xi}$ the Lipschitz constant of $h_{\xi}$. Then as $T\rightarrow \infty$, we have that
\[\lim_{T \rightarrow \infty } \hat \mu_{T, f, \eta} (\xi) = \phi_{f, \eta}(\xi)= \lim_{Y \rightarrow \infty} \frac1Y \int_{\eta_0}^Y e^{-i \xi E[f, g_{y, \eta}]} \, dy.\]
To show that this is continuous at $\xi = 0$, we will show that $\lim_{\xi \rightarrow 0} \phi_{f, \eta}(\xi) = \phi_{f, \eta}(0) = 1$. 

By the uniform bound \eqref{E-unif-bound}, we know that $|E[f, g_{y, \eta}] | \leq M_{f, \eta}$ for some $M_{ f, \eta}>0$
which is uniform in $y$. Then, we have that for $|\xi|< \frac{1}{M_{f, \eta}}$, 
\[e^{i\xi E[f, g_{y, \eta}]} - 1 \ll M_{f, \eta} |\xi|.\]
Thus, we establish the  bound
\[| \lim_{Y \rightarrow \infty} \frac1Y \int_{\eta_0}^Y \left(e^{ i \xi E[f, g_{y, \eta}]}-1\right) \, dy| \leq \lim_{Y \rightarrow \infty} \frac1Y \int_{\eta_0}^Y M_{f, \eta} |\xi| \, dy = M_{f, \eta} |\xi|\]
which converges to zero as $|\xi| \rightarrow 0$. This proves that $\phi_{f, \eta}(\xi)$ is continuous at $\xi = 0$.

Then by L\'evy Continuity Theorem, there exists a measure $\alpha_{f, \eta}$ such that $\hat \alpha_{f, \eta}  = \phi_{f, \eta}$.  But note that $\phi_{f, \eta}(\xi) = \hat \mu_{f, \eta}(\xi)$ by definition, and by  Fourier inversion, $\alpha_{f, \eta} = \mu_{f, \eta}$. Thus $\mu_{f, \eta}$ is a measure, which completes the proof.
\end{proof}

\begin{cor}\label{mu-mean} For any $\eta>0$, $\mu_{f, \eta}$, and $\Gamma$ as in Theorem~\ref{Eydistribution},
	
	1) The distribution $\mu_{f, \eta}$ has mean $b_{f, \eta}$ defined in \eqref{b-f-eta-background}.
	
	2) The distribution $\mu_{f,  \eta}$ is compactly supported on the interval 
	\begin{equation*}
	 [b_{f, \eta} - 2\sum_{s, p} m_\Gamma(\pi_{is,p})  |\hat f(-p)| |c_{s, \eta}| - C_{\Gamma, f, \eta}, b_{f, \eta} + 2 \sum_{s,p} m_\Gamma(\pi_{is,p}) |\hat f(-p)| |c_{s, \eta}| + C_{\Gamma, f, \eta}],
	\end{equation*}
	where $C_{\Gamma, f, \eta}$ is a constant which accounts for the error term in \eqref{E-unif-bound}.
\end{cor}
\begin{proof}
The mean value of $\mu_{f, \eta}$ is obtained by considering the function $h(x) = x$ in Theorem \ref{Eydistribution}. The compact support results from the fact that the distribution is of the function
\[E[f, g_{y, \eta}] = 2 \sum_{\substack{s \neq 0\\ p\neq 0}} m_\Gamma(\pi_{is,p})  \Re\left(\hat f(-p) e^{isy}c_{s, \eta}\right)  + b_{f, \eta} + \mathrm{O}_{\Gamma, f, \eta}(1),\]
 and the constant $C_{\Gamma, f, \eta}$ accounts for the error $\mathrm{O}_{\Gamma, f, \eta}(1)$ from \eqref{E-unif-bound}.
\end{proof}

\section{Linear Independence Hypothesis}\label{LI-section}

Recall that in Corollary \ref{ET-measure-LI}, we were able to explicate the distribution of $E^{(T)}[f,g_{y, \eta}]$ when all distinct spectral parameters are linearly independent. This will allow us to prove an explicit formula for the measure $\mu_{f, \eta}$ using its Fourier transform, $\hat \mu_{f, \eta}$. For distinct parameters $s$ and $p$ where $m_\Gamma(\pi_{is, p})$ is non-zero and $s,p \neq 0$, we now fix an indexing where
\[\max\{|s_1|, |p_1|\} \leq \max\{|s_2|, |p_2|\} \leq ... \]
Note that by the spectral bound from Proposition \ref{spectral-bound}, there are finitely many parameters with $\max\{|s|, |p|\} \leq T$, so such an indexing is possible; this can also be made consistent with the indexing for $E^T[f,g_{y, \eta}]$ in Section \ref{Bias-Distribution-Section}.

\begin{prop}\label{mu-hat-formula} Let $\Gamma < \PSL$ be a discrete, co-compact, torsion-free subgroup. Suppose that all distinct non-zero spectral parameters of $L^2(\Gamma \backslash G)$ are linearly independent. Then,
	\begin{equation}\label{hat-mu-equation}
\hat \mu_{f,\eta}(\xi) =  e^{-i \xi b_{f, \eta}} \prod_{j =1}^\infty J_0\left(\xi 2 m_{\Gamma}(\pi_{i {s_j}, p_j}) | \hat f(-p_j) c_{s_j,\eta}|\right)
	\end{equation}
	where the product is over the representations $\pi_{i s_j, p}$ in $L^2(\Gamma \backslash G)$ with multiplicities $m_\Gamma(\pi_{i {s_j}, p_j})$, $b_{f,\eta}$ is defined as in \eqref{b-f-eta-background}, $c_{s_j, \eta}$ is defined as in \eqref{c-s-eta-def}, and $J_0$ is a Bessel function of the first kind, which is real and even.
\end{prop}

\begin{proof}
	Let $\hat \mu_{f,\eta}(\xi) = \int_{\mathbb{R}} e^{-i \xi x} d \mu_{f,\eta}(x)$. By Theorem \ref{Eydistribution}, the measure $\mu_{f, \eta}$ is the limit of the measures $\mu_{T, f, \eta}$, and so 
	\[\hat \mu_{f,\eta}(\xi)  = \lim_{T \rightarrow \infty}  \int_{\mathbb{R}} e^{-i \xi x}\, d \mu_{T,f, \eta}(x).\]
	
	Then by Corollary \ref{ET-measure-LI}, this is equal to 
	\begin{align*}
	&\lim_{n \rightarrow \infty} \int_0^1... \int_0^1 e^{-i \xi(2 \sum_{j =1}^n m_\Gamma(\pi_{i{s_j},p_j}) \Re(\hat f(-p_j) e^{2 \pi i x_j}c_{{s_j}, \eta}) +b_{f, \eta}))}\, dx_1 ... dx_n\\
	&= e^{-i \xi b_{f, \eta}} \prod_{j = 1}^\infty \int_0^1 e^{-i \xi 2  m_\Gamma(\pi_{i{s_j},p_j})  \Re(\hat f(-p_j) c_{s_j, \eta} e^{2 \pi i x_j})} \, dx_j = e^{-i \xi b_{f, \eta}} \prod_{j = 1}^\infty \hat \mu_{f, \eta,j}(\xi),
	\end{align*}
	where $\hat \mu_{f,\eta,j}(\xi) =  \int_0^1 e^{-i \xi 2 m_\Gamma(\pi_{i{s_j},p_j}) \Re( \hat f(-p_j) c_{s_j, \eta} e^{2 \pi i x_j})} \, dx_j$. %(We note that this limit may contain additional values of $n$ -- rather than just the cutoff values for a particular $T$ -- but since this limit exists, it will also equal the limit of the subsequence as $T \rightarrow \infty$.)
	We note that the result from Corollary \ref{ET-measure-LI}, as well as the bound \eqref{mu-T-convergence}, apply to the sum up to an arbitrary $n$ (rather than a cutoff at a particular $T$), so this will converge.
	We compute that $\Re(\hat f(-p_j) c_{s_j, \eta} e^{2 \pi i x_j}) = |\hat f(-p_j) c_{{s_j}, \eta}| \cos(2 \pi x_j + \arg(\hat f(-p_j)) + \arg c_{{s_j},\eta})$. Then, using a change of variables we have
	\begin{align*}
	\hat\mu_{f,\eta,j}(\xi) &= \int_0^1 e^{-i \xi 2  m_\Gamma(\pi_{i{s_j},p_j})  |\hat f(-p_j) c_{{s_j}, \eta}| \cos( 2 \pi x_j + \arg(\hat f(-p_j)) + \arg c_{{s_j}, \eta})} \, d x_j\\
	&= \frac{1}{2\pi} \int_0^{2\pi} e^{-i \xi 2 m_\Gamma(\pi_{i{s_j},p_j}) |\hat f(-p_j) c_{{s_j}, \eta}| \cos(  x) } \, d x.
	\end{align*}
Then, using symmetry, we conclude that
\begin{align*}
\hat\mu_{f,\eta,j}(\xi) &=  \frac{1}{2\pi} \int_0^{2\pi} e^{-i \xi 2 m_\Gamma(\pi_{i{s_j},p_j}) |\hat f(-p_j) c_{{s_j}, \eta}| \cos(x)} dx \\
&= \frac{1}{\pi} \int_{0}^\pi \cos(\xi 2 m_\Gamma(\pi_{i{s_j},p_j})  |\hat f(-p_j) c_{{s_j}, \eta}| \cos(x)) dx \\
&= J_0\left(\xi 2 m_\Gamma(\pi_{i{s_j},p_j}) |\hat f(-p_j) c_{{s_j}, \eta}|\right),
\end{align*}
	where $J_0$ is a Bessel function of the first kind in the integral form given by \cite[Eq. 10.9.1]{NIST}.
\end{proof}

\begin{remark}\label{LI-relax-remark} The linear independence condition in Corollary \ref{ET-measure-LI}, Proposition \ref{mu-hat-formula}, and Theorem \ref{LI-Theorem} -- that all distinct spectral parameters are linearly independent -- is necessary when $\hat f(-p) \neq 0$ for all $p \neq 0$. However, when $\hat f(-p) = 0$ for a particular $p$, then the corresponding multiplicities disappear from \eqref{E-f-g-def}, and we need only the linear independence of distinct spectral parameters $s$ corresponding to $\pi_{is, p}$  in $L^2(\Gamma \backslash G)$ with $\hat f(p) \neq 0$. 

For example, consider $f(\theta) = \cos(\theta)$, which has $\hat f(1) =\hat f(-1) =\frac{1}{2}$ and all other Fourier coefficients vanish. Therefore, the multiplicities of $\pi_{is, p}$ for $p \neq 1, -1$ do not appear in the sum \eqref{E-f-g-def}, and thus we only need to consider the linear independence of distinct spectral parameters $s$ of $\pi_{is, 1} \simeq \pi_{-is, -1}$ in $L^2( \Gamma \backslash G)$.
\end{remark}

When the linear independence condition is met, we can then directly compute the measure $\mu_{f,\eta}$. 

\begin{theorem}\label{LI-Theorem} Let $\Gamma < \PSL$ be a discrete, co-compact, torsion-free subgroup. Suppose all distinct spectral parameters of $L^2(\Gamma \backslash G)$ are linearly independent. The distribution $\mu_{f,\eta}$ from Theorem \ref{Eydistribution} has $d \mu_{f,\eta}(x) = p_{f,\eta}(x) \, dx$ where the distribution function is
	\[ p_{f,\eta}(x) = \frac{1}{( 2\pi)^2} \int_{-\infty}^\infty \prod_{j = 1}^\infty  J_0\left(\xi 2 m_{\Gamma}(\pi_{i {s_j}, p_j}) | \hat f(-p_j) c_{s_j,\eta}|\right) \cos(\xi (x- b_{f, \eta})) \,d \xi\]
	where $J_0$ is a Bessel function of the first kind and $c_{s, \eta}$ are defined as in \eqref{c-s-eta-def}. The distribution $\mu_{f,\eta}$ is symmetric about $b_{f,\eta}$ (defined in \eqref{b-f-eta-background}), and a bias exists if and only if $b_{f, \eta} \neq 0$.
\end{theorem}

\begin{proof}
	The distribution function $p_{f, \eta}(x)$ can be recovered by taking the Fourier transform of $\hat \mu_{f,\eta}(\xi)$ and using Proposition \ref{mu-hat-formula}. This give us
	\begin{align*}
	p_{f,\eta}(x) &= \frac{1}{(2\pi)^2} \int_{-\infty}^{\infty} \hat \mu_{f,\eta}(\xi) e^{i \xi x} \,d \xi\\
	&=\frac{1}{(2\pi)^2} \int_{-\infty}^{\infty} e^{-i \xi b_{f, \eta}} \prod_{j = 1}^\infty  J_0\left(\xi 2 m_{\Gamma}(\pi_{i {s_j}, p_j}) | \hat f(-p_j) c_{s_j,\eta}|\right) e^{ i \xi x}\, d \xi \\
	&= \frac{1}{( 2\pi)^2} \int_{-\infty}^\infty \prod_{j = 1}^\infty J_0\left(\xi 2 m_{\Gamma}(\pi_{i {s_j}, p_j}) | \hat f(-p_j) c_{s_j,\eta}|\right) \cos\left(\xi (x- b_{f, \eta})\right)\, d \xi
	\end{align*}
	which is symmetric about $b_{f, \eta}$ since $J_0$ is real and even.
\end{proof}

\begin{remark} We remark that $p_{f, \eta}$ is non-negative. This is apparent since $\mu_{T, f, \eta}$ are non-negative,  $\lim_{T \rightarrow \infty} \mu_{T, f, \eta} = \mu_{f, \eta}$, and $d \mu_{f, \eta}= p_{f, \eta}(x)\, dx$. However, we can also observe non-negativity from this specific construction.  In Proposition \ref{mu-hat-formula}, we show that $\hat \mu_{f, \eta}$ can be written as the product \eqref{hat-mu-equation}. Then the Fourier inversion formula can also be constructed as the convolution
	\[p_{f, \eta}(x) = \left(\mu_{f, \eta, 1} \ast \mu_{f, \eta, 2} \ast  .... \right)(x- b_{f, \eta}),\]
where the factor $e^{-i \xi b_{f, \eta}}$ contributes to a shift by $b_{f, \eta}$.

Each $\hat \mu_{f, \eta, j}$ was defined as
\[\hat \mu_{f,\eta,j}(\xi) =  \int_0^1 e^{-i \xi 2 m_\Gamma(\pi_{i{s_j},p_j}) \Re( \hat f(-p_j) c_{s_j, \eta} e^{2 \pi i x})} \, dx.\]
This can also be constructed from the measure obtained by applying the Kronecker-Weyl Theorem to 
\[\lim_{X \rightarrow \infty} \frac{1}{X} \int_{0}^X h(2 m_\Gamma(\pi_{i{s_j},p_j}) \Re( \hat f(-p_j) c_{s_j, \eta} e^{2 \pi i x})) \, dx = \int_0^1 h(x) \,d \mu_{f, \eta, j}(x),\]
where the left hand side simplifies to $\int_0^1  h(2 m_\Gamma(\pi_{i{s_j},p_j}) \Re( \hat f(-p_j) c_{s_j, \eta} e^{2 \pi i x})) \, dx$ by periodicity. Choosing $h_{\xi}(x) = e^{-i \xi x}$ yields that $\hat \mu_{f, \eta, j}(\xi) = \int_0^1 e^{-i \xi x} \, d \mu_{f, \eta, j}(x)$ is in fact the Fourier transform of $\mu_{f, \eta, j}$. In addition, $\mu_{f, \eta, j}$ arises as the pushforward measure $\mu_{f, \eta, j}(B) =  |w_{f, \eta,j}^{-1}(B)|$ where $w_{f, \eta, j}(x)=2 m_\Gamma(\pi_{i{s_j},p_j}) \Re( \hat f(-p_j) c_{s_j, \eta} e^{2 \pi i x})$, and the Lebesgue measure  is non-negative.
\end{remark}

\section{Linear Dependence Example}\label{LD-section}

The linear independence hypothesis in Theorem \ref{LI-Theorem} is not satisfied when there is an arithmetic progression of spectral parameters, which is known to be true for certain classes of dihedral forms on arithmetic hyperbolic 3-manifolds. Although there are explicit descriptions of dihedral forms on non-compact hyperbolic surfaces (see \cite{Luo}, \cite{Huang}, for example), we were not able to find an explicit construction of dihedral forms on compact hyperbolic 3-manifolds in the literature. Thus, we present a general approach to the construction of dihedral forms on compact arithmetic hyperbolic 3-manifolds (which we refer to as co-compact dihedral forms) and provide an explicit example to show that dihedral representations $\pi_{\nu, p}$ can appear with  integral weight $p \neq 0$ which interferes with the bias distribution. In particular, even if we specialize to the function $f(\theta) = \cos(\theta)$ on the holonomy, which then only requires linear independence the spectral parameters of representations $\pi_{\nu, 1}$ in Theorem \ref{LI-Theorem} (see Remark \ref{LI-relax-remark}), we find an example where there is an arithmetic progression of spectral parameters. We construct these forms from Hecke characters whose angular frequencies lie in arithmetic progressions, which correspond to arithmetic progressions of spectral parameters.

\subsection{General construction of  co-compact dihedral forms}

Let $D$ be a division quaternion algebra over a number field $F$, $K$ a quadratic extension of $F$, $\chi$ an non-trivial Hecke character of $K$, and $\pi(\chi)$ the induced dihedral representation on $\GL_2(\mathbb{A}_F)$. If $\pi(\chi_v)$ is a discrete series representation at each place of ramification $v$ of $D$, then by the Jacquet-Langlands correspondence (see \cite{JL}, \cite[Theorem 10.5]{Gelbart}), $\pi(\chi)$ lifts to a representation $\pi$ on $D^\times( \mathbb{A}_F)$. The co-compact subgroup $\Gamma$ of $\PSL$, with level depending on local representations $\pi(\chi_v)$ (and thus on the local characters $\chi_v$, which are fixed in our construction), is obtained by embedding an order of the quaternion algebra $D$ into $M_2(\mathbb{C})$ and taking the group of units, and the dihedral representation of $D$ then corresponds to a dihedral automorphic form on $\Gamma  \backslash \PSL$, a compact manifold. The dihedral forms may have some non-trivial nebentypus, but we may obtain dihedral forms of trivial nebentypus by passing to a fixed finite index subgroup $\Gamma_0<\Gamma$.

To obtain dihedral forms, it suffices to construct a Hecke character $\chi$ of $K$ such that $\pi(\chi_{v})$ is discrete series at each place of ramification of $D$. In particular, we can ensure that $\pi(\chi_v)$ is a supercuspidal representation when there is a unique place $w$ of $K$ over $v$ where $\chi_w$ is not invariant under Gal$(K_w/F_v)$; in particular, $w$ is inert over $v$ \cite{Schmidt}. This then implies that $\pi(\chi_v)$ is a discrete series representation; see \cite{Sally}.

\subsection{Example of co-compact dihedral forms}

We will now provide one fully explicit example of such a character. We consider the field $F = \mathbb{Q}(i)$ and $K = \mathbb{Q}(i, \sqrt{3}) = \mathbb{Q}(\zeta_{12})$. Then $K$ has two complex embeddings $v_1$ and $v_2$. The group of units of $K$ is $\mathbb{Z} \times \mathbb{Z}/12 \mathbb{Z}$; the generators are the fundamental unit $2 + \sqrt{3}$ and the twelfth root of unity $\zeta_{12} = e^{2 \pi i /12}$. Throughout this section we use $\zeta_n$ to denote the $n^{th}$ root of unity $e^{2\pi i/n}$.

We use  two rational primes $5$ and $17$  which are $1\bmod 4$. Thus they split as $5 = p_1 \bar p_1$ and $17 = p_2 \bar p_2$ where $p_1 = 2+i$, $p_2 = 4+i$. Then $|F_{p_1}| = 5$ and $|F_{p_2}| = 17$, and we note that $5$ and $17$ are $2 \mod 3$, so $3$ is not a square in $F_{p_1}$ and $F_{p_2}$, and $p_1$ and $p_2$ stay inert in $K/F$. Then let $D$ be the unique quaternion algebra ramified over $F$ at $p_1$ and $p_2$ and split at all other places; see \cite[Theorem 7.3.6]{MR} for the existence and uniqueness of this quaternion algebra.

We define the multiplicative function $\chi = \prod_v \chi_v$ on the ideles $I_K$ as the product of characters on each completion $K_{v}^\times$. We will now construct $\chi_v$ for $v = p_1, p_2$, and the infinite place $v_1$; at all other places we take $\chi_v$ to be trivial. For $\chi$ to be a Hecke character, it must be trivial on the generators $\zeta_{12}$ and $2 + \sqrt{3}$ of the units in $K$.

Since $p_1$ remains a prime in $K$, $K_{p_1} = F_{p_1}(\sqrt{3})$ is a proper quadratic extension; let $\mathcal{O}_{p_1}$ be its ring of integers and let $\hat{\mathfrak{p}}_1$ be the principal ideal generated by $p_1$. Recall that the multiplicative group of $K_{p_1}$ decomposes into
\[K_{p_1}^* = (p_1) \times \mu_{p_1} \times U^{(1)} \]
where $(p_1)$ is the group of integral powers of $p_1$,  $U^{(1)}$ is the group $1 + \mathfrak{p}_1$ of principal units, and $\mu_{p_1}$ is a set of representatives of $(\mathcal{O}_{p_1}/ \hat{\mathfrak{p}}_1)^\times$ (which will be chosen to form a group) \cite[Prop. II.5.3]{Neu99}.  We choose $\chi_{p_1}$ to be trivial on $(p_1)$ and $U^{(1)}$ and will now explicitly define it on $\mathcal{O}_{p_1}/\hat{\mathfrak{p}}_1$.

There exists an isomorphism $\mathcal{O}_{p_1}/\hat{\mathfrak{p}}_1 \cong \mathcal{O}/\mathfrak{p}_1$, where on the right hand side we are no longer considering the p-adic closure \cite[Prop. II.4.3]{Neu99}. In this case, $\mathcal{O} = \mathbb{Z}[\zeta_{12}]$ is the ring of integers of $K$, and $\mathfrak{p}_1$ is the maximal ideal of $2+i$ in $\mathcal{O}$.

This can be made completely explicit. Since $\zeta_{12}$ has minimal polynomial  $x^4 - x^2+1$, the ring of integers of $K$ is
\[\mathbb{Z}[\zeta_{12}] = \{ a_0 + a_1 \zeta_{12} + a_2 \zeta_{12}^2+a_{3} \zeta_{12}^{3}: a_i \in \mathbb{Z}\}\]
where we can reduce by $\zeta_{12}^4 = \zeta_{12}^2-1$. Note that this contains $2 + i = 2 + \zeta_{12}^3$. Now we want to consider the maximal ideal generated by $2 + i$. This will contain $(2+i)(2-i) = 5$. In addition, this contains $2 + i = 2+ \zeta_{12}^3$ and $\zeta_{12}(2 + i) = 2 \zeta_{12} + \zeta_{12}^4 = \zeta_{12}^2+ 2\zeta_{12}-1$. Therefore, in the quotient space $\mathbb{Z}[\zeta_{12}] /\mathfrak{p}_1$, we can reduce with $\zeta_{12}^2 \equiv (1 - 2 \zeta_{12}) \bmod p_1$. Thus the space can be reduced down to
\[\mathbb{Z}[\zeta_{12}]/ \mathfrak{p}_1 = \{ [a_0 + a_1 \zeta_{12}]: a_i \in \mathbb{Z}/5\mathbb{Z}\}\]
which is a field with $5^2 = 25$ elements (since $\mathfrak{p}_1$ is a maximal ideal). The multiplicative group of order $24$ is the coefficient group $\mu_{p_1}$ for the $p_1$-adic numbers $K_{p_1}$.

By direct calculation, one can verify that $\zeta+3$ is a generator of this group. We choose the character $\chi_{p_1}(\zeta_{12}+3) = \zeta_{24}$, which has the property that $\chi_{p_1}(-1) = -1$. We also compute that in this quotient ring, the generating units of $K$ are the products $2 + \sqrt{3} = 4 + 2 \zeta_{12} \equiv (\zeta_{12}+3)^{16} \bmod p_1$, and $\zeta_{12} \equiv (\zeta_{12} + 3)^{20} \bmod p_1$. Therefore, the values of the local character at the generating units of $K$ are 
\[\chi_{p_1}(\zeta_{12}) = \zeta_{24}^{20} = \zeta_{6}^5,\] 
\[\chi_{p_1}(2 + \sqrt{3}) = (\zeta_{24})^{16} = \zeta_{3}^2.\]

Similarly for $p_2 = 4 + i$, we define the Hecke character $\chi_{p_2}$ on the coefficient group $\mathcal{O}_{p_2}/\hat{\mathfrak{p}}_2 \cong \mathcal{O}/\mathfrak{p}_2$. In this case, $\mathfrak{p}_2$ contains $(4 + i) (4 - i) = 17$ and $\zeta_{12}(4 + \zeta_{12}^3) = \zeta_{12}^2 + 4 \zeta_{12}-1$. Thus we can reduce the quotient space to
\[\mathbb{Z}[\zeta_{12}]/ \mathfrak{p}_2 = \{ [a_0 + a_1 \zeta_{12}]: a_i \in \mathbb{Z}/17\mathbb{Z}\},\]
which is a field of order $17^2 = 289$. The multiplicative group is a cyclic group of order $288$.

To construct the character, we choose the generator $7 + \zeta_{12}$, which we confirm is a generator using Sage. We then set $\chi_{p_2}(7 + \zeta_{12}) = \zeta_{288}$, which guarantees that $\chi_{p_2}(-1) = -1$. We also calculate that in the quotient ring, $(7 + \zeta_{12})^{24} \equiv \zeta_{12} \bmod p_2$ and $(7 + \zeta_{12})^{208} \equiv 6 + 2\zeta \equiv(2 + \sqrt{3}) \bmod p_2$. Therefore,
\[\chi_{p_2}(\zeta_{12}) = \zeta_{288}^{24} = \zeta_{12},\]
\[\chi_{p_2}(2 + \sqrt{3}) = \zeta_{288}^{208} = \zeta_{18}^{13}.\]

At the infinite place $v_1$, we can choose a character $\chi_{v_1}(z) = |z|^{it_1} (\frac{z}{|z|})^{k_1}$, where $t_1\in \mathbb{R}$ and $k_1 \in \mathbb{Z}$ will be determined. (Note that this corresponds to the character $\chi_{\nu,p}$ defined in \eqref{characters-T} with $\nu = i t_1$ and $p = k_1$, and so the induced representation $\pi(\chi_{v_1}) \cong \pi_{it_1, k_1}$ in the notation of previous sections.) To ensure that $\chi$ is trivial on the generators, these must satisfy the following conditions:
\[\chi_{p_1}(\zeta_{12}) \chi_{p_2}(\zeta_{12}) \chi_{\nu_1}(\zeta_{12}) =  \zeta_6^5  \zeta_{12} \zeta_{12}^{k_1} = \zeta_{12}^{11+k_1} = 1, \]
\[ \chi_{p_1}(2 + \sqrt{3}) \chi_{p_2}(2 + \sqrt{3}) \chi_{\nu_1}(2 + \sqrt{3}) = \zeta_{3}^2 \zeta_{18}^{13} |2 + \sqrt{3}|^{i t_1} = 1.\]
The first equation is satisfied when $k_1 \equiv 1 \bmod 12$. The second equation is satisfied by 
\[t_1 = \frac{2 \pi(n + 11/18)}{\log|2 + \sqrt{3}|}\]
for $n \in \mathbb{Z}$. This yields an arithmetic progression of solutions (not centered at $0$)  for each $k \equiv 1 \bmod 12$. Note that this character depends on the choices of generators for each quotient ring, and other arithmetic progressions of spectral parameters could potentially be obtained by choosing different generators.

Recall that $K_{p_1} = F_{p_1}(\sqrt{3})$ is a quadratic extension. Therefore, the non-trivial element of $\text{Gal}(K_{p_i}/F_{p_i})$ will map $\sqrt{3} \mapsto - \sqrt{3}$. Since $\chi_{p_i}(-1) = -1$, and therefore $\chi_{p_i}(-\sqrt{3}) = - \chi_{p_i}(\sqrt{3})$, $\chi_{p_i}$ is not invariant under the Galois group $\text{Gal}(K_{p_i}/F_{p_i})$, and thus $\pi(\chi_v)$ is a discrete series representation at each place of ramification $v$ of $D$ (namely, $v \in p_1, p_2$).

\bibliographystyle{amsalpha}
\bibliography{biblio}

\end{document}